\documentclass{amsart}
\usepackage{amsmath}
\usepackage{amssymb}
\usepackage{amsthm}

\DeclareMathOperator{\Id}{Id}

\DeclareMathOperator{\tr}{tr}

\DeclareMathOperator{\Vol}{Vol}

\DeclareMathOperator{\End}{End}

\DeclareMathOperator{\Imaginary}{Im}
\DeclareMathOperator{\Real}{Re}

\DeclareMathOperator{\tf}{tf}

\newcommand{\db}{\partial_b}
\newcommand{\dbbar}{\overline{\partial}_b}

\newcommand{\ow}{\overline{w}}

\newcommand{\hg}{\widehat{g}}

\newcommand{\hI}{\widehat{I}}

\newcommand{\hQ}{\widehat{Q}}

\newcommand{\hX}{\widehat{X}}

\newcommand{\hmI}{\widehat{\mathcal{I}}}

\newcommand{\htheta}{\widehat{\theta}}
\newcommand{\hzeta}{\widehat{\zeta}}
\newcommand{\homega}{\widehat{\omega}}

\newcommand{\lp}{\langle}
\newcommand{\rp}{\rangle}
\newcommand{\lv}{\lvert}
\newcommand{\rv}{\rvert}

\newcommand{\contr}{\lrcorner}

\newcommand{\mC}{\mathcal{C}}

\newcommand{\mE}{\mathcal{E}}

\newcommand{\mI}{\mathcal{I}}

\newcommand{\mO}{\mathcal{O}}
\newcommand{\mP}{\mathcal{P}}

\newcommand{\mR}{\mathcal{R}}
\newcommand{\mS}{\mathcal{S}}
\newcommand{\mT}{\mathcal{T}}
\newcommand{\mU}{\mathcal{U}}
\newcommand{\mV}{\mathcal{V}}

\newcommand{\bC}{\mathbb{C}}

\newcommand{\bN}{\mathbb{N}}

\newcommand{\bR}{\mathbb{R}}

\newcommand{\suchthat}{\mathrel{}\middle|\mathrel{}}

\def\sideremark#1{\ifvmode\leavevmode\fi\vadjust{\vbox to0pt{\vss
 \hbox to 0pt{\hskip\hsize\hskip1em
 \vbox{\hsize3cm\tiny\raggedright\pretolerance10000
 \noindent #1\hfill}\hss}\vbox to8pt{\vfil}\vss}}}

\newcommand{\comment}[1]{}

\newcommand{\chern}{\mathrm{ch}}

\numberwithin{equation}{section}

\usepackage[noabbrev,capitalize]{cleveref}

\newtheorem{theorem}{Theorem}[section]
\newtheorem{lemma}[theorem]{Lemma}
\newtheorem{proposition}[theorem]{Proposition}
\newtheorem{corollary}[theorem]{Corollary}
\newtheorem{conjecture}[theorem]{Conjecture}

\theoremstyle{definition}
\newtheorem{remark}[theorem]{Remark}
\newtheorem{definition}[theorem]{Definition}

\theoremstyle{theorem}
\newtheorem{step}{Step}

\begin{document}

\title[$\mI^\prime$-curvatures and the Hirachi conjecture]{$\mI^\prime$-curvatures in higher dimensions and the Hirachi conjecture}
\author{Jeffrey S. Case}
\thanks{JSC was supported by a grant from the Simons Foundation (Grant No.\ 524601)}
\address{109 McAllister Building \\ Penn State University \\ University Park, PA 16802}
\email{jscase@psu.edu}
\author{Yuya Takeuchi}
\thanks{YT was supported by JSPS Research Fellowship for Young Scientists
	and JSPS KAKENHI Grant Number JP19J00063}
\address{Department of Mathematics \\ Graduate School of Science \\ Osaka University
	\\ 1-1 Machikaneyama-cho, Toyonaka, Osaka 560-0043, Japan}
\email{yu-takeuchi@cr.math.sci.osaka-u.ac.jp}
\keywords{secondary CR invariant; Hirachi conjecture}
\subjclass[2010]{Primary 32V05}
\begin{abstract}
 We construct higher-dimensional analogues of the $\mI^\prime$-curvature of Case and Gover in all CR dimensions $n\geq2$.  Our $\mI^\prime$-curvatures all transform by a first-order linear differential operator under a change of contact form and their total integrals are independent of the choice of pseudo-Einstein contact form on a closed CR manifold.  We exhibit examples where these total integrals depend on the choice of general contact form, and thereby produce counterexamples to the Hirachi conjecture in all CR dimensions $n\geq2$.
\end{abstract}
\maketitle

\section{Introduction}
\label{sec:intro}

The $Q^\prime$-curvature of a pseudo-Einstein manifold~\cite{CaseYang2012,Hirachi2013} has many formal similarities to the (critical) $Q$-curvature in conformal geometry~\cite{Branson1995}.  These similarities begin with how the $Q^\prime$- and $Q$-curvatures transform under a conformal rescaling of the contact form and the metric, respectively.  If $\theta$ and $\htheta=e^\Upsilon\theta$ are pseudo-Einstein contact forms on a $(2n+1)$-dimensional CR manifold, then
\begin{equation}
 \label{eqn:Qprime-transformation}
 e^{(n+1)\Upsilon}\hQ^\prime = Q^\prime + P^\prime(\Upsilon) + \frac{1}{2}P(\Upsilon^2) \equiv Q^\prime + P^\prime(\Upsilon) \mod \mP^\perp ,
\end{equation}
where $P^\prime$ is the $P^\prime$-operator~\cite{CaseYang2012,Hirachi2013}, $P$ is the (critical) CR GJMS operator~\cite{GoverGraham2005}, and $\mP^\perp$ is the $L^2$-orthogonal complement to the space $\mP$ of CR pluriharmonic functions.  Similarly, if $g$ and $\hg=e^{2\Upsilon}g$ are Riemannian metrics on a $2n$-dimensional manifold, then
\begin{equation}
 \label{eqn:Q-transformation}
 e^{2n\Upsilon}\hQ = Q + P(\Upsilon),
\end{equation}
where $P$ is the (critical) GJMS operator~\cite{GJMS1992}.  Importantly, the operators appearing in \cref{eqn:Qprime-transformation,eqn:Q-transformation} are formally self-adjoint and annihilate constants.  In particular, the total $Q^\prime$-curvature is a \emph{global secondary CR invariant} --- that is, it is independent of the choice of pseudo-Einstein contact form, if one exists, on a closed CR manifold --- and the total $Q$-curvature is a global conformal invariant.  Moreover, explicit formulae for the $Q^\prime$-curvature of the round CR sphere~\cite{CaseGover2013,Takeuchi2018} and the $Q$-curvature of the round sphere~\cite{Branson1995} imply that these global invariants are nontrivial.

For $(2n+1)$-dimensional CR manifolds which can be realized as the boundary of a bounded strictly pseudoconvex domain in $\bC^{n+1}$, the total $Q^\prime$-curvature is a global biholomorphic invariant of the domain.
The Burns--Epstein invariant~\cite{BurnsEpstein1988,BurnsEpstein1990c} is also a global biholomorphic invariant of such a domain.  Marugame~\cite{Marugame2016} gave an alternative realization of the Burns--Epstein invariant as the boundary term in a Gauss--Bonnet--Chern formula for the domain.  When $n=1$, the total $Q^\prime$-curvature agrees, up to a multiplicative constant, with the Burns--Epstein invariant~\cite{CaseYang2012,Hirachi2013}.  When $n=2$, the total $Q^\prime$-curvature and the Burns--Epstein invariant are linearly independent, but an explicit relationship in terms of global secondary CR invariants is known~\cite{CaseGover2013,HirachiMarugameMatsumoto2015}.

The analogue of the above paragraph in conformal geometry is the relationship between the total $Q$-curvature and the Euler characteristic.  It is well-known that the Gauss--Bonnet formula identifies the Euler characteristic of a closed surface with the total $Q$-curvature, up to multiplicative constant.  The Gauss--Bonnet--Chern formula in dimension four gives an explicit identity relating the Euler characteristic, the total $Q$-curvature, and the $L^2$-norm of the Weyl tensor~\cite{BransonOrsted1991b}.  Similarly, the Gauss--Bonnet--Chern formula in dimension six gives an explicit identity relating the Euler characteristic, the total $Q$-curvature, and total integrals of local conformal invariants~\cite{Graham2000}.  More generally, Alexakis~\cite{Alexakis2012} proved that if $I$ is any natural Riemannian scalar invariant whose total integral is a conformal invariant on any closed $2n$-dimensional manifold, then there is a constant $c\in\bR$ such that
\[ I = cQ + \text{(local conformal invariant)} + \text{(divergence)} . \]
Together with the close relationship between the $Q^\prime$- and $Q$-curvatures, Alexakis' result motivated Hirachi~\cite{Hirachi2013} to pose the following conjecture:

\begin{conjecture}[Hirachi conjecture]
 \label{conj:strong_hirachi}
 Let $I$ be a natural pseudohermitian scalar invariant whose total integral is a secondary CR invariant.  Then there is a constant $c\in\bR$ such that
 \begin{equation}
  \label{eqn:strong_hirachi}
  I = cQ^\prime + \textup{(local CR invariant)} + \textup{(divergence)} .
 \end{equation}
\end{conjecture}

\cref{conj:strong_hirachi} is true~\cite{Hirachi2013} in CR dimension $n=1$; i.e.\ if $I$ is a natural pseudohermitian scalar invariant whose total integral is a secondary CR invariant on all closed CR three-manifolds, then $I$ is of the form of \cref{eqn:strong_hirachi}.  However, \cref{conj:strong_hirachi} is false~\cite{CaseGover2013,ReiterSon2019} in CR dimension $n=2$.  The purpose of this article is to show that it is false in all CR dimensions $n\geq2$ by producing a large collection of counterexamples.  To motivate our results, we first describe in more detail what is known when $n=2$.

Let $(M^5,T^{1,0},\theta)$ be a pseudohermitian manifold of CR dimension $n=2$.  Case and Gover~\cite{CaseGover2013} studied two invariants.  First, they proved that
\[ X_\alpha := -iS_{\alpha\bar\beta\gamma\bar\sigma}V^{\bar\beta\gamma\bar\sigma} + \frac{1}{4}\nabla_\alpha\lv S_{\gamma\bar\sigma\delta\bar\rho}\rv^2 \]
is a CR invariant $(1,0)$-form of weight $-2$, where $S_{\alpha\bar\beta\gamma\bar\sigma}$ is the Chern tensor and, in general CR dimension $n$,
\[ V_{\alpha\bar\beta\gamma} := \frac{i}{n}\nabla^{\bar\sigma}S_{\alpha\bar\beta\gamma\bar\sigma} . \]
Case and Gover further showed that if $(M^5,T^{1,0})$ admits a pseudo-Einstein contact form, then $[\xi]=4\pi^2c_2(T^{1,0})$, where
\[ \xi := 2\Real X_\alpha\theta\wedge\theta^\alpha\wedge d\theta . \]
By observing~\cite{CaseGover2013,Takeuchi2020} that $c_2(T^{1,0})=0$ in $H^4(M;\bR)$, they conclude that $\Real\nabla^\alpha X_\alpha$ is orthogonal to $\mP$.  Second, they proved that the $\mI^\prime$-curvature,
\[ \mI^\prime := -\frac{1}{8}\Delta_b\lv S_{\alpha\bar\beta\gamma\bar\sigma}\rv^2 + \lv V_{\alpha\bar\beta\gamma}\rv^2 + \frac{1}{2}P\lv S_{\alpha\bar\beta\gamma\bar\sigma}\rv^2, \]
where $P:=\frac{1}{2(n+1)}R$ is a constant multiple of the pseudohermitian scalar curvature, is such that
\[ e^{3\Upsilon}\hmI^\prime = \mI^\prime + 2\Real X_\alpha\Upsilon^\alpha \]
for any $\htheta=e^\Upsilon\theta$, where $\hmI^\prime$ is defined in terms of $\htheta$.  These facts imply that the total $\mI^\prime$-curvature is a global secondary CR invariant; in fact, the Burns--Epstein invariant is a linear combination of the total $Q^\prime$- and $\mI^\prime$-curvatures~\cite{CaseGover2013}.  By computing on nonspherical real ellipsoids, Reiter and Son~\cite{ReiterSon2019} then showed that the $\mI^\prime$-curvature is not a linear combination of a local CR invariant and a divergence, thereby disproving \cref{conj:strong_hirachi} in CR dimension two.

In this article we construct analogues of $X_\alpha$ and $\mI^\prime$ in all CR dimensions $n\geq2$.  To that end, let $\delta_{\alpha_1\dotsm\alpha_n}^{\beta_1\dotsm\beta_n}$ denote the generalized Kronecker delta and let $\Phi_{\alpha_1\dotsm\alpha_n}^{\beta_1\dotsm\beta_n}$ be an invariant polynomial of degree $n$; in particular,
\[ \Phi_{\alpha_{\sigma(1)}\dotsm\alpha_{\sigma(n)}}^{\beta_{\sigma(1)}\dotsm\beta_{\sigma(n)}} = \Phi_{\alpha_1\dotsm\alpha_n}^{\beta_1\dotsm\beta_n} \]
for all elements $\sigma\in S_n$ of the symmetric group on $n$ elements.  Define
\begin{equation}
 \label{eqn:Xn}
 X_\alpha^\Phi := i(\mS^\Phi)_\alpha{}^\beta{}_\mu{}^\nu V_\beta{}^\mu{}_\nu - \frac{1}{n^2}\nabla_\alpha c_\Phi(S),
\end{equation}
where
\begin{align}
 (\mS^\Phi)_\alpha{}^\beta{}_\mu{}^\nu & := \delta_{\alpha\alpha_2\dotsm\alpha_n}^{\beta\beta_2\dotsm\beta_n}\Phi_{\mu\mu_2\dotsm\mu_n}^{\nu\nu_2\dotsm\nu_n}S_{\beta_2}{}^{\alpha_2}{}_{\nu_2}{}^{\mu_2}\dotsm S_{\beta_n}{}^{\alpha_n}{}_{\nu_n}{}^{\mu_n} , \\
 c_\Phi(S) & := (S^\Phi)_\alpha{}^\beta{}_\mu{}^\nu S_\beta{}^\alpha{}_\nu{}^\mu .
\end{align}
Taking $\Phi_{\alpha_1\alpha_2}^{\beta_1\beta_2}=\delta_{\alpha_1}^{\beta_2}\delta_{\alpha_2}^{\beta_1}$ recovers the definitions of Case and Gover~\cite{CaseGover2013}.

Our first result is that $X_\alpha^\Phi$ is CR invariant:

\begin{theorem}
 \label{thm:X-invariant}
 Let $(M^{2n+1},T^{1,0},\theta)$ be a pseudoherimitian manifold, let $\Phi$ be an invariant polynomial of degree $n$, and let $X_\alpha^\Phi$ be given by~\cref{eqn:Xn}.  Then $X_\alpha^\Phi$ is a CR invariant $(1,0)$-form of weight $-n$; i.e.
 \[ e^{n\Upsilon}\hX_\alpha^\Phi = X_\alpha^\Phi \]
 for all $\htheta=e^\Upsilon\theta$, where $\hX_\alpha^\Phi$ is defined in terms of $\htheta$. In particular, $\Real \nabla^{\alpha} X_{\alpha}^{\Phi}$ is a local CR invariant of weight $-n-1$.
\end{theorem}

This follows by a direct computation using the CR invariance of the Chern tensor and the simple transformation formula for $V_{\alpha\bar\beta\gamma}$; see~\cref{sec:invariance} for details.

Now define the $\mI_\Phi^\prime$-curvature of $(M^{2n+1},T^{1,0},\theta)$ by
\begin{multline}
 \label{eqn:mIn}
 \mI_{\Phi}^\prime := \frac{1}{n^3}\Delta_b c_\Phi(S) - \frac{2}{n^2}Pc_\Phi(S) \\
  + (\mT^\Phi)_\alpha{}^\beta{}_{\mu_1}{}^{\nu_1}{}_{\mu_2}{}^{\nu_2}\left((n-1) V_\beta{}^{\mu_1}{}_{\nu_1} V^\alpha{}_{\nu_2}{}^{\mu_2} - S_\beta{}^\alpha{}_{\nu_1}{}^{\mu_1}U_{\nu_2}{}^{\mu_2}\right)
\end{multline}
where
\[ \bigl(\mT^\Phi\bigr)_\alpha{}^\beta{}_{\mu_1}{}^{\nu_1}{}_{\mu_2}{}^{\nu_2} := \delta_{\alpha\alpha_3\dotsm\alpha_n}^{\beta\beta_3\dotsm\beta_n}\Phi_{\mu_1\dotsm\mu_n}^{\nu_1\dotsm\nu_n}S_{\beta_3}{}^{\alpha_3}{}_{\nu_3}{}^{\mu_3}\dotsm S_{\beta_n}{}^{\alpha_n}{}_{\nu_n}{}^{\mu_n} \]
and $U_{\alpha\bar\beta}$ is related to $\nabla^\gamma V_{\alpha\bar\beta\gamma}$; see \cref{sec:bg} for the precise definition.  Our second result is that the transformation formula for $\mI_\Phi^\prime$ is given by the first-order linear differential operator $\Real X_\alpha^\Phi\nabla^\alpha$.

\begin{theorem}
 \label{thm:mI-transformation}
 Let $(M^{2n+1},T^{1,0},\theta)$ be a pseudohermitian manifold, let $\Phi$ be an invariant polynomial of degree $n$, and let $\mI_{\Phi}^\prime$ be given by~\cref{eqn:mIn}.  For any $\Upsilon\in C^\infty(M)$, it holds that
 \begin{equation}
  \label{eqn:mI-transformation}
  e^{(n+1)\Upsilon}\hmI_\Phi^\prime = \hmI_\Phi^\prime + 2\Real X_\alpha^\Phi\Upsilon^\alpha ,
 \end{equation}
 where $\hmI_\Phi^\prime$ is defined by $\htheta:=e^{\Upsilon}\theta$ and $X_\alpha^\Phi$ is given by~\cref{eqn:Xn}.
\end{theorem}

This follows by a direct computation using the CR invariance of the Chern tensor and the simple transformation formulae for $V_{\alpha\bar\beta\gamma}$ and $U_{\alpha\bar\beta}$; see~\cref{sec:invariance} for details.

Our third result is that the total $\mI_\Phi^\prime$-curvature is a secondary CR invariant.

\begin{theorem}
 \label{thm:secondary_invariant}
 Let $(M^{2n+1},T^{1,0})$ be a closed CR manifold which admits a pseudo-Einstein contact form $\theta$ and let $\Phi$ be an invariant polynomial of degree $n$.  If $\htheta$ is also a pseudo-Einstein contact form, then
 \[ \int_M \hmI_\Phi^\prime\,\htheta\wedge d\htheta^n = \int_M \mI_\Phi^\prime\,\theta\wedge d\theta^n . \]
\end{theorem}

Recall that if $\theta$ is pseudo-Einstein, then $e^\Upsilon\theta$ is pseudo-Einstein if and only if $\Upsilon$ is a CR pluriharmonic function~\cite{Lee1988}.  Thus~\cref{thm:secondary_invariant} is equivalent to the claim that $\Real\int X_\alpha^\Phi\Upsilon^\alpha=0$ for all CR pluriharmonic functions $\Upsilon$.  We prove this in the same spirit as the proof of Case and Gover~\cite{CaseGover2013} in the case $n=2$:

The CR invariance of $X_\alpha^\Phi$ implies that $\xi^\Phi:=2\Real X_\alpha^\Phi\theta\wedge\theta^\alpha\wedge d\theta^{n-1}$ is a CR invariant $2n$-form of weight $0$.  A straightforward consequence of Lee's Bianchi identities~\cite{Lee1988} implies that $\xi^\Phi$ is closed.  We show that if $(M^{2n+1},T^{1,0})$ admits a pseudo-Einstein contact form, then $[\xi^\Phi]$ is proportional to the characteristic class $c_\Phi(T^{1,0})\in H^{2n}(M;\bR)$ determined by $\Phi$; see \cref{prop:chern}.  An observation of Takeuchi~\cite{Takeuchi2020} implies that $c_\Phi(T^{1,0})=0$.  \Cref{thm:secondary_invariant} then follows from the fact that $\Real\int X_\alpha^\Phi\Upsilon^\alpha$ equals, up to a multiplicative constant, the evaluation of the cup product $[\xi^\Phi]\cup[d_b^c\Upsilon]:=[\xi^\Phi\wedge d_b^c\Upsilon]$ on the fundamental class of $M$ whenever $\Upsilon\in\mP$.  Note that Marugame~\cite{Marugame2019} showed that one can relax the assumption that $(M^{2n+1},T^{1,0})$ admits a pseudo-Einstein contact form to $c_1(T^{1,0})=0$ in $H^2(M;\bR)$.

Our last result is that there is a large variety of choices of invariant polynomials $\Phi$ for which the total $\mI_\Phi^\prime$-curvature gives a counterexample to \cref{conj:strong_hirachi}.  Our strategy is as follows:

Suppose \cref{conj:strong_hirachi} holds.  Let $\Phi$ be an invariant polynomial of degree $n$.  On the one hand, there exists a constant $c$, depending only on $\Phi$, such that
\begin{equation*}
	\mathcal{I}_{\Phi}^{\prime}
	= c Q^{\prime} + \text{(local CR invariant)} + \text{(divergence)}.
\end{equation*}
Consider the round CR sphere $(S^{2 n + 1}, T^{1, 0}, \theta)$.
In this case,
$\mathcal{I}_{\Phi}^{\prime}$ and any local CR invariant are identically zero,
but $Q^{\prime}$ is a nonzero constant~\cite{CaseGover2013,Takeuchi2018}.
Integrating implies that $c = 0$,
and hence $\mathcal{I}_{\Phi}^{\prime}$ can be written as the sum of a local CR invariant and a divergence.
In particular,
the total $\mathcal{I}_{\Phi}^{\prime}$-curvature is a global CR invariant. 
On the other hand,
under a general conformal change $\widehat{\theta} = e^{\Upsilon} \theta$, \cref{thm:mI-transformation} implies that
\begin{equation*}
	\int_{M} \widehat{\mathcal{I}}_{\Phi}^{\prime} \widehat{\theta} \wedge d \widehat{\theta}^{n}
	= \int_{M} \mathcal{I}_{\Phi}^{\prime} \theta \wedge d \theta^{n}
		- 2 \int_{M} (\Real \nabla^{\alpha} X^{\Phi}_{\alpha}) \Upsilon \theta \wedge d \theta^{n}.
\end{equation*}
One arrives at a contradiction by finding an example of $\Phi$ and $(M, T^{1, 0})$
such that $\Real \nabla^{\alpha} X^{\Phi}_{\alpha} \neq 0$;
see \cref{lem:construction-of-counterexample}.

Let $\varsigma=(\varsigma_1,\dotsc,\varsigma_n)\in\bN^n$ be such that $\varsigma_1+2\varsigma_2+\dotsm+n\varsigma_n=n$ and let $\Phi(\varsigma)$ be the invariant polynomial of degree $n$ defined by
\begin{equation}
 \label{eqn:Phi_varsigma}
 \Phi(\varsigma)_{\alpha_1\dotsm\alpha_n}^{\beta_1\dotsm\beta_n}A_{\beta_1}{}^{\alpha_1}\dotsm A_{\beta_n}{}^{\alpha_n} = \prod_{k=1}^n \left(\tr A^k\right)^{\varsigma_k}
\end{equation}
for $\tr A^k:=A_{\gamma_1}{}^{\gamma_2}A_{\gamma_2}{}^{\gamma_3}\dotsm A_{\gamma_k}{}^{\gamma_1}$.  Our first counterexamples come from considering $\Phi(\varsigma)$ on perturbations of the round CR sphere.

\begin{theorem}
\label{thm:counterexample-via-perturb-intro}
 For $n \geq 2$, there exists a perturbation of the round CR sphere in $\mathbb{C}^{n + 1}$ such that $\Real \nabla^{\alpha} X^{\Phi(\varsigma)}_{\alpha}$ is not identically zero for any $\varsigma$ with $\varsigma_{1} = 0$.
 In particular, the $\mI_{\Phi(\varsigma)}^{\prime}$-curvature gives a counterexample to the Hirachi conjecture.
\end{theorem}

This result follows from \cref{thm:hirachi_ellipsoid}, where we compute variations of $\Real \nabla^\alpha X_\alpha^{\Phi(\varsigma)}$ for a deformation of the round CR $(2n+1)$-sphere.  This deformation is in the direction of a real ellipsoid, and gives a local (in the space of CR structures on $S^{2n+1}$) analogue of the computation of $\Real\nabla^\alpha X_\alpha^{\Phi(0,1)}$ on $5$-dimensional ellipsoids by Reiter and Son~\cite{ReiterSon2019}.

Second, we consider the case that $\Phi=(n)$ is the generalized Kronecker delta on $n$ variables.

\begin{theorem}[= \cref{thm:counterexamples-via-calabi-yau}]
\label{thm:counterexamples-via-calabi-yau-intro}
 For $n \geq 2$, there exists a closed $(2n + 1)$-dimensional pseudo-Einstein manifold $(M, T^{1, 0}, \theta)$ such that
\begin{equation*}
R_{\alpha \bar{\beta}} = 0,
\qquad
A_{\alpha \beta} = 0,
\qquad
\Real \nabla^{\alpha} X^{(n)}_{\alpha} \neq 0.
\end{equation*}
 In particular, the $\mI_{(n)}^{\prime}$-curvature gives a counterexample to the Hirachi conjecture.
\end{theorem}

This is a consequence of degenerations of Ricci-flat K\"{a}hler metrics. There exists a smooth family of Ricci-flat K\"{a}hler metrics on a certain Calabi--Yau manifold whose curvature concentrates along some complex submanifolds. Together with the Gauss--Bonnet--Chern formula, this implies that for many members of this family, there is a circle bundle which is a Ricci-flat Sasakian manifold with $\Real\nabla^\alpha X_\alpha^{(n)} \neq 0$. These examples have the benefit of being significantly easier to compute.

\cref{thm:counterexample-via-perturb-intro,thm:counterexamples-via-calabi-yau-intro} imply that the total $\mI_\Phi^\prime$-curvatures are nontrivial on general pseudohermitian manifolds.  In fact, the total $\mI_\Phi^\prime$-curvature are nontrivial secondary CR invariants.  We prove this by computing the total $\mI_{\Phi}^{\prime}$-curvatures of the boundaries of locally homogeneous Reinhardt domains.

\begin{theorem}[= \cref{thm:I-prime-curvature-for-Reinhardt-domains}]
\label{thm:I-prime-curvature-for-Reinhardt-domains-intro}
	For $r > 0$,
	let $M_{r}$ be the boundary of the bounded Reinhardt domain
	\begin{equation*}
		\Omega_{r}
		= \left\{ w = (w^{0}, \dotsc , w^{n}) \in \mathbb{C}^{n + 1} \suchthat
			\sum_{j = 0}^{n} (\log |w^{j}|)^{2} < r^{2} \right\}.
	\end{equation*}
	The total $\mathcal{I}^{\prime}_{\Phi(\varsigma)}$-curvature $\overline{\mathcal{I}}^{\prime}_{\Phi(\varsigma)}$ of $M_{r}$
	is given by
	\begin{equation*}
		\overline{\mathcal{I}}^{\prime}_{\Phi(\varsigma)}
		= - (n !)^{2} \Vol(S^{n}(1))  \left( \frac{\pi}{(n + 1)r} \right)^{n + 1}
			\prod_{k = 1}^{n} [(n + 2)(1 - (n + 2)^{k - 1})]^{\varsigma_{k}},
	\end{equation*}
	where $\Vol(S^{n}(1))$ is the volume of the unit sphere in $\mathbb{R}^{n + 1}$.
\end{theorem}

If $\varsigma_{1} = 0$, then
the total $\mathcal{I}^{\prime}_{\Phi(\varsigma)}$-curvature of $M_{r}$ is of the form $C r^{- n - 1}$ for $C$ a nonzero constant depending only on $n$ and $\varsigma$.
In particular,
the total $\mathcal{I}^{\prime}_{\Phi(\varsigma)}$-curvature is
a nontrivial secondary CR invariant when $\varsigma_{1} = 0$.
Since two bounded strictly pseudoconvex domains in $\mathbb{C}^{n + 1}$ are biholomorphic
if and only if their boundaries are CR equivalent~\cite{Fefferman1974}, we obtain the following corollary.

\begin{corollary}
	The domains $\Omega_{r}$ and $\Omega_{r^{\prime}}$ are biholomorphic
	if and only if $r = r^{\prime}$.
\end{corollary}

This corollary was proven using different global CR invariants
by Burns and Epstein~\cite{BurnsEpstein1988} for $n = 1$,
Marugame~\cite{Marugame2016} for $n = 2$,
and Reiter and Son~\cite{ReiterSon2019} for any dimension.
In other words,
we give another proof of the result of Reiter--Son
by using $\mathcal{I}^{\prime}$-curvatures.
Note that this corollary also follows from a result by Sunada~\cite{Sunada1978}
for general bounded Reinhardt domains.

Finally, we note that Marugame~\cite{Marugame2019} has independently established~\cref{thm:X-invariant,thm:mI-transformation,thm:secondary_invariant} in the same generality that we consider, and also discussed the nontriviality of the total $\mI_\Phi^\prime$-curvatures.  His proof of the CR invariance of $X_\alpha^\Phi$ uses the tractor calculus in a way analogous to the work of Case and Gover~\cite{CaseGover2013}, while his proof that the total $\mI_\Phi^\prime$-curvature is a secondary CR invariant uses a tractor-based proof that $\xi^\Phi$ represents a multiple of $c_\Phi(T^{1,0})$.  His work produces other global secondary CR invariants, but their explicit realization as total integrals of local pseudohermitian invariants remains unknown.  His work does not determine whether the invariants constructed give counterexamples to \cref{conj:strong_hirachi}.

This article is organized as follows.  In \cref{sec:bg} we collect some necessary background material.  In \cref{sec:chern} we give some equivalent realizations of characteristic classes in terms of various $\End(T^{1,0})$-valued two-forms.  In~\cref{sec:invariance} we prove~\cref{thm:X-invariant,thm:mI-transformation,thm:secondary_invariant}.  In~\cref{sec:hirachi} we further discuss our strategy to disprove \cref{conj:strong_hirachi}.  In \cref{sec:perturb} we prove \cref{thm:counterexample-via-perturb-intro}.  In \cref{sec:calabi-yau} we prove \cref{thm:counterexamples-via-calabi-yau-intro}.  In \cref{sec:reinhardt} we prove \cref{thm:I-prime-curvature-for-Reinhardt-domains-intro}.  In \cref{sec:conclusion} we propose a weaker version of \cref{conj:strong_hirachi} and discuss it in the context of the $\mI^\prime$-curvature.
\section{Background}
\label{sec:bg}

In this section we collect necessary background material.

\subsection{CR and pseudohermitian manifolds}

A \emph{CR manifold} $(M^{2n+1},T^{1,0})$ is a real $(2n+1)$-dimensional manifold $M^{2n+1}$ together with a rank $n$ distribution $T^{1,0}\subset TM\otimes\bC$ such that $[T^{1,0},T^{1,0}]\subset T^{1,0}$ and $T^{1,0}\cap T^{0,1}=\{0\}$ for $T^{0,1}:=\overline{T^{1,0}}$.  We assume throughout that $M$ is orientable.  We say that $(M^{2n+1},T^{1,0})$ is \emph{strictly pseudoconvex} if there exists a real one-form $\theta$ on $M$ such that $\ker\theta=\Real T^{1,0}$ and $-i\,d\theta(Z,\overline{W})$ defines a positive definitive Hermitian form on $T^{1,0}$.  We call such a $\theta$ a \emph{contact form}.  Note that contact forms are determined up to multiplication by a positive function.

Given a CR manifold $(M^{2n+1},T^{1,0})$ and a smooth (complex-valued) function $f\in C^\infty(M;\bC)$, we denote by $\partial_bf$ the restriction of $df$ to $T^{1,0}$; likewise $\dbbar f:=df\rv_{T^{0,1}}$.  A \emph{CR function} is a function $f\in C^\infty(M;\bC)$ such that $\dbbar f=0$.  A \emph{CR pluriharmonic function} is a (real-valued) function $u\in C^\infty(M)$ such that locally $u=\Real f$ for some CR function $f$; i.e.\ for every $p\in M$, there is a neighborhood $U$ of $p$ and a CR function $f\in C^\infty(U;\bC)$ such that $u\rv_U=\Real f$.  Denote by $\mP$ the space of CR pluriharmonic functions.  We emphasize that the notion of a CR pluriharmonic function is defined without reference to a choice of contact form.  An infinitesimal characterization of CR pluriharmonic functions via differential operators has been given by Lee~\cite[Propositions~3.3 and~3.4]{Lee1988}.

A \emph{pseudohermitian manifold} $(M^{2n+1},T^{1,0},\theta)$ is a triple consisting of a strictly pseudoconvex CR manifold $(M^{2n+1},T^{1,0})$ and a choice of contact form.  The \emph{Reeb vector field} $T$ is the unique vector field such that $\theta(T)=1$ and $d\theta(T,\cdot)=0$.  Denote by $T^{\ast(1,0)}$ the subbundle of $T^\ast M\otimes\bC$ which annihilates $T^{0,1}$ and $T$.  Set $T^{\ast(0,1)}:=\overline{T^{\ast(1,0)}}$.  The \emph{Tanaka--Webster connection} of $(M^{2n+1},T^{1,0},\theta)$ is defined as follows: Let $\{\theta^\alpha\}_{\alpha=1}^n$ be an \emph{admissible coframe of $T^{\ast(1,0)}$}; i.e.\ $\theta^\alpha\in T^{\ast(1,0)}$ for all $\alpha=1,\dotsc,n$ and $\{\theta^1,\dotsc,\theta^n,\theta^{\bar 1},\dotsc,\theta^{\bar n},\theta\}$ forms a basis for $T^\ast M\otimes\bC$, where $\theta^{\bar\beta}:=\overline{\theta^\beta}$.  It follows that there is a positive definite Hermitian matrix $h_{\alpha\bar\beta}$ such that
\[ d\theta = ih_{\alpha\bar\beta}\theta^\alpha \wedge \theta^{\bar\beta} . \]
We use $h_{\alpha\bar\beta}$ and its inverse $h^{\alpha\bar\beta}$ to lower and raise indices as needed.  The \emph{connection one-forms} $\omega_\alpha{}^\beta$ associated to $\{\theta^\alpha\}$ are uniquely determined by
\begin{align*}
 d\theta^\alpha & = \theta^\beta\wedge\omega_\beta{}^\alpha + \theta \wedge \tau^\alpha, & \tau^\alpha & = A^{\alpha}{}_{\bar\beta}\theta^{\bar\beta} , \\
 dh_{\alpha\bar\beta} & = \omega_{\alpha\bar\beta}+\omega_{\bar\beta\alpha}, & A_{\alpha\beta} & = A_{\beta\alpha} .
\end{align*}
The tensor $A_{\alpha\beta}$ is the \emph{pseudohermitian torsion}.  Note that
\begin{equation}
 \label{eqn:tautheta}
 \theta^\gamma\wedge \tau_\gamma = 0 .
\end{equation}
The connection one-forms determine the Tanaka--Webster connection by $\nabla\theta=0$ and $\nabla\theta^\alpha=-\omega_\gamma{}^\alpha\otimes\theta^\gamma$.  The \emph{curvature two-forms} $\Pi_\alpha{}^\beta$ are the $\End(T^{1,0})$-valued two-forms
\begin{equation}
 \label{eqn:Pi}
 \Pi_\alpha{}^\beta := d\omega_\alpha{}^\beta - \omega_\alpha{}^\gamma\wedge\omega_\gamma{}^\beta .
\end{equation}
The \emph{pseudohermitian curvature} $R_{\alpha\bar\beta\gamma\bar\sigma}$ is the coefficient of the $(1,1)$-part of $\Pi_\alpha{}^\beta$; i.e.
\[ \Pi_\alpha{}^\beta \equiv R_\alpha{}^\beta{}_{\gamma\bar\sigma}\theta^\gamma\wedge\theta^{\bar\sigma} \mod \theta, \theta^\alpha\wedge\theta^\gamma, \theta^{\bar\beta}\wedge\theta^{\bar\sigma} . \]
The \emph{pseudohermitian Ricci tensor} $R_{\alpha\bar\beta}$ and \emph{pseudohermitian scalar curvature} $R$ are defined by taking traces in the usual way; i.e.\ $R_{\alpha\bar\beta}:=R_{\alpha\bar\beta\gamma}{}^\gamma$ and $R:=R_\gamma{}^\gamma$.  We say that $(M^{2n+1},T^{1,0},\theta)$, $n\geq2$, is \emph{pseudo-Einstein} if $R_{\alpha\bar\beta}=\frac{1}{n}Rh_{\alpha\bar\beta}$.  If $(M^{2n+1},T^{1,0},\theta)$ is pseudo-Einstein, then $c_1(T^{1,0})$ vanishes in $H^2(M;\bR)$~\cite[Proposition~D]{Lee1988}.

The pseudohermitian torsion, pseudohermitian curvature, and covariant derivatives are all tensorial.  We may thus use abstract index notation to denote tensors.  Specifically, unbarred Greek superscripts denote factors of $T^{1,0}$, barred Greek superscripts denote factors of $T^{0,1}$, unbarred Greek subscripts denote factors of $T^{\ast(1,0)} M$, and barred Greek subscripts denote factors of $T^{\ast(0,1)} M$.  For example, $C_{\alpha\bar\beta}{}^\gamma$ denotes a section of $T^{\ast(1,0)}\otimes T^{\ast(0,1)}\otimes T^{1,0}$.  We keep the notation $\nabla$ to denote covariant derivatives.  For example, $\nabla_\rho C_{\alpha\bar\beta}{}^\gamma$ denotes the $(1,0)$-part of the covariant derivative of $C_{\alpha\bar\beta}{}^\gamma$. When clear by context, we use subscripts to denote covariant derivatives of a function $u\in C^\infty(M;\bC)$; e.g.\ $u_{\alpha\bar\beta}:=\nabla_{\bar\beta}\nabla_\alpha u$. 
We use $\nabla_0$ to denote covariant derivatives in the direction of the Reeb vector field.

The \emph{sublaplacian} $\Delta_b$ of a pseudohermitian manifold is the operator
\[ \Delta_b u := u_\gamma{}^\gamma + u^\gamma{}_\gamma \]
for all $u\in C^\infty(M;\bC)$.  Recall that if $(M^{2n+1},T^{1,0},\theta)$ is closed, then $\ker\Delta_b$ equals the space of locally constant functions.

We require three curvature tensors naturally associated to a pseudohermitian manifold $(M^{2n+1},T^{1,0},\theta)$, all of which appear as components of the CR tractor curvature~\cite{GoverGraham2005}.

The first curvature tensor we need is the \emph{Chern tensor}
\[ S_{\alpha\bar\beta\gamma\bar\sigma} := R_{\alpha\bar\beta\gamma\bar\sigma} - P_{\alpha\bar\beta}h_{\gamma\bar\sigma} - P_{\alpha\bar\sigma}h_{\gamma\bar\beta} - P_{\gamma\bar\beta}h_{\alpha\bar\sigma} - P_{\gamma\bar\sigma}h_{\alpha\bar\beta}, \]
where $P_{\alpha\bar\beta}:=\frac{1}{n+2}(R_{\alpha\bar\beta}-Ph_{\alpha\bar\beta})$ is the \emph{CR Schouten tensor} and $P:=\frac{1}{2(n+1)}R$ is its trace.  The relevance of the Chern tensor to CR geometry is that if $n\geq2$, then $S_{\alpha\bar\beta\gamma\bar\sigma}=0$ if and only if $(M^{2n+1},T^{1,0})$ is locally CR equivalent to the round CR $(2n+1)$-sphere.  Importantly, the Chern tensor is symmetric and trace-free:
\begin{align*}
 S_{\alpha\bar\beta\gamma\bar\sigma} & = S_{\alpha\bar\sigma\gamma\bar\beta} = S_{\gamma\bar\beta\alpha\bar\sigma} , \\
 S_{\alpha\bar\beta\gamma}{}^\gamma & = 0 .
\end{align*}

The second curvature tensor we need is
\[ V_{\alpha\bar\beta\gamma} := \nabla_{\bar\beta}A_{\alpha\gamma} + i\nabla_\gamma P_{\alpha\bar\beta} - iT_\gamma h_{\alpha\bar\beta} - 2iT_\alpha h_{\gamma\bar\beta} , \]
where $T_\alpha:=\frac{1}{n+2}(\nabla_\alpha P - i\nabla^\gamma A_{\alpha\gamma})$.  This tensor is a divergence of the Chern tensor:
\begin{equation}
 \label{eqn:divS}
 \nabla^{\bar\sigma}S_{\alpha\bar\beta\gamma\bar\sigma} = -niV_{\alpha\bar\beta\gamma} ;
\end{equation}
see~\cite[Lemma~2.2]{CaseGover2013}.  Importantly, $V_{\alpha\bar\beta\gamma}$ is symmetric and trace-free:
\begin{align*}
 V_{\alpha\bar\beta\gamma} & = V_{\gamma\bar\beta\alpha} , \\
 V_\alpha{}^\gamma{}_\gamma & = 0 . 
\end{align*}

The third curvature tensor we need is
\[ U_{\alpha\bar\beta} := \nabla_{\bar\beta}T_\alpha + \nabla_\alpha T_{\bar\beta} + P_\alpha{}^\rho P_{\rho\bar\beta} - A_{\alpha\rho}A^\rho{}_{\bar\beta} + Sh_{\alpha\bar\beta} , \]
where $S\in C^\infty(M;\bC)$ is such that $U_\gamma{}^\gamma=0$.  This tensor is closely related to a divergence of $V_{\alpha\bar\beta\gamma}$:
\begin{equation}
 \label{eqn:divV}
 \nabla^\gamma V_{\alpha\bar\beta\gamma} = niU_{\alpha\bar\beta} - iS_{\alpha\bar\beta\gamma\bar\sigma}P^{\gamma\bar\sigma} ;
\end{equation}
see~\cite[Lemma~2.2]{CaseGover2013}.

In addition to the well-known CR invariance of the Chern tensor, we need to know how the tensors $V_{\alpha\bar\beta\gamma}$ and $U_{\alpha\bar\beta}$ transform under change of contact form.  To that end, given a natural pseudohermitian tensor $B$ on $(M^{2n+1},T^{1,0},\theta)$ which is homogeneous of degree $k$ in $\theta$ --- that is, $B_{c\theta}=c^kB_\theta$ for all constants $c>0$ --- define the \emph{conformal linearization $D_\theta B$ of $B$ at $\theta$} by
\[ D_\theta B(\Upsilon) := \left.\frac{\partial}{\partial t}\right|_{t=0} e^{-kt\Upsilon}B_{e^{t\Upsilon}\theta} \]
for all $\Upsilon\in C^\infty(M)$.  It is clear that $D_\theta B(1)=0$.  One easily checks that $D_\theta$ extends to a derivation on the space of natural homogeneous pseudohermitian tensors.  By a simple integration argument (cf.\ \cite{Branson1985}), the tensor $B$ is a \emph{local CR invariant of weight $k$} --- that is,
\[ e^{-k\Upsilon}B_{\htheta} = B_\theta \]
for all contact forms $\theta$ and $\htheta=e^\Upsilon\theta$ --- if and only if $D_\theta B\equiv0$.

The following lemma collects the well-known~\cite{GoverGraham2005,Lee1988} conformal linearizations of the CR Schouten tensor, the Chern tensor, and the Tanaka--Webster connection, as well as the needed conformal linearizations of $V_{\alpha\bar\beta\gamma}$ and $U_{\alpha\bar\beta}$.  Note that these conformal linearizations can also be deduced from the CR invariance of the curvature of the CR tractor connection~\cite{GoverGraham2005}.

\begin{lemma}
 \label{lem:variations}
 Let $(M^{2n+1},T^{1,0},\theta)$ be a pseudohermitian manifold and let $\Upsilon\in C^\infty(M)$.  Then
 \begin{align*}
  D_\theta P_{\alpha\bar\beta}(\Upsilon) & = -\frac{1}{2}\left(\Upsilon_{\alpha\bar\beta} + \Upsilon_{\bar\beta\alpha}\right) , \\
  D_\theta S_{\alpha\bar\beta\gamma\bar\sigma}(\Upsilon) & = 0 , \\
  D_\theta V_{\alpha\bar\beta\gamma}(\Upsilon) & = iS_{\alpha\bar\beta\gamma}{}^\rho \Upsilon_\rho , \\
  D_\theta U_{\alpha\bar\beta}(\Upsilon) & = iV_{\bar\sigma\alpha\bar\beta}\Upsilon^{\bar\sigma} - iV_{\alpha\bar\beta\gamma}\Upsilon^\gamma .
 \end{align*}
 If $f$ is a local scalar CR invariant of weight $k$, then
 \[ D_\theta \nabla_\alpha f(\Upsilon) = kf\Upsilon_\alpha . \]
 If $\omega_\alpha$ is a natural pseudohermitian $(1,0)$-form which is homogeneous of degree $k$ in $\theta$, then
 \begin{align*}
  D_\theta \nabla_\gamma\omega_\alpha(\Upsilon) & = (k-1)\omega_\alpha\Upsilon_\gamma - \Upsilon_\alpha\omega_\gamma + \nabla_\gamma\left(D_\theta\omega_\alpha(\Upsilon)\right) , \\
  D_\theta \nabla_{\bar\beta}\omega_\alpha & = k\omega_\alpha\Upsilon_{\bar\beta} + \Upsilon^\gamma\omega_\gamma h_{\alpha\bar\beta} + \nabla_{\bar\beta}\left(D_\theta\omega_\alpha(\Upsilon)\right) .
 \end{align*}
\end{lemma}

\begin{proof}
 All but the formulae for $D_\theta V_{\alpha\bar\beta\gamma}$ and $D_\theta U_{\alpha\bar\beta}$ follow from~\cite[Proposition~2.3, Equation~(2.7), Equation~(2.8)]{GoverGraham2005}.  Computing the conformal linearization of both sides of \cref{eqn:divS,eqn:divV} yields the claimed formulae for $D_\theta V_{\alpha\bar\beta\gamma}$ and $D_\theta U_{\alpha\bar\beta}$, respectively.
\end{proof}

The following consequences of the Bianchi identities are useful in studying $X_\alpha^\Phi$ and related objects.

\begin{lemma}
 \label{lem:bianchi}
 Let $(M^{2n+1},T^{1,0},\theta)$ be a pseudohermitian manifold.  Then
 \begin{align}
  \label{eqn:bianchiS} \nabla_{[\alpha} S_{\beta]}{}^\rho{}_\gamma{}^\sigma & = iV_\gamma{}^\rho{}_{[\alpha}\delta_{\beta]}^\sigma + iV_\gamma{}^{\sigma}{}_{[\alpha}\delta_{\beta]}^\rho , \\
  \label{eqn:bianchiV} \nabla_{[\alpha} V_{\beta]}{}^\gamma{}_\rho & = -S_\rho{}^\gamma{}_{[\alpha}{}^\sigma A_{\beta]\sigma} + iQ_{\rho[\alpha}\delta_{\beta]}^\gamma , \\
  \label{eqn:S0} \nabla_0S_{\alpha\bar\beta\gamma\bar\sigma} & = \nabla_{\bar\sigma}V_{\alpha\bar\beta\gamma} + \nabla_\gamma V_{\bar\beta\alpha\bar\sigma} - iS_{\gamma\bar\sigma\alpha}{}^\rho P_{\rho\bar\beta} - iS_{\gamma\bar\sigma\rho\bar\beta}P_\alpha{}^\rho \\
  \notag & \qquad + iU_{\alpha\bar\sigma}h_{\gamma\bar\beta} - iU_{\gamma\bar\beta}h_{\alpha\bar\sigma} .
 \end{align}
 where $T_{[\alpha\gamma]}:=\frac{1}{2}(T_{\alpha\gamma}-T_{\gamma\alpha})$ and $Q_{\alpha\gamma}:=i\nabla_0A_{\alpha\gamma} - 2i\nabla_\gamma T_\alpha + 2P_\alpha{}^\rho A_{\rho\gamma}$.
\end{lemma}

\begin{proof}
 \cref{eqn:bianchiS} follows from~\cite[Equation~(2.7)]{Lee1988}.  \cref{eqn:bianchiV} follows from~\cite[Equations~(2.9) and~(2.14)]{Lee1988}.  \cref{eqn:S0} follows from~\cite[Equation~(2.8)]{Lee1988}.
\end{proof}

\subsection{Sasakian manifolds}

We recall some facts about Sasakian manifolds;
see~\cite{BoyerGalicki2008} for a comprehensive introduction.
A \emph{Sasakian manifold}
is a pseudohermitian manifold $(M, T^{1,0}, \theta)$
with pseudohermitian torsion identically zero,
or equivalently,
the Reeb vector field $T$ preserves the CR structure $T^{1, 0}$.

A typical example of a Sasakian manifold
is the circle bundle associated with a negative holomorphic line bundle.
Let $Y$ be an $n$-dimensional complex manifold
and $(L, h)$ a Hermitian holomorphic line bundle over $Y$
such that $\omega = - i \Theta_{h} = 2^{- 1} d d^{c} \log h$
defines a K\"{a}hler metric on $Y$,
where $d^{c} = i(\overline{\partial} - \partial)$.
Now we consider the circle bundle
\begin{equation*}
	M := \left\{ v \in L \suchthat h(v, v) = 1 \right\},
\end{equation*}
which is a real hypersurface in $L$.
The one-form $\theta := 2^{- 1} d^{c} \log h |_{M}$
is a connection one-form of the principal $S^{1}$-bundle $p \colon M \to Y$
and satisfies $d \theta = p^{\ast} \omega$.
Moreover,
the natural CR structure $T^{1, 0}$ on $M$
coincides with the horizontal lift of the holomorphic tangent bundle $T^{1, 0} Y$ of $Y$ with respect to $\theta$.
Since $\omega$ defines a K\"{a}hler metric,
we have
\begin{equation*}
	- i d \theta (Z, \overline{Z})
	= - i \omega(p_{\ast} Z, p_{\ast} \overline{Z})
	> 0
\end{equation*}
for all nonzero $Z \in T^{1, 0}$.
Hence $(M, T^{1, 0}, \theta)$ is a pseudohermitian manifold of dimension $2 n + 1$.
We call this triple the \emph{circle bundle associated with} $(Y, L, h)$.
Note that the Reeb vector field $T$ with respect to $\theta$
is a generator of the $S^{1}$-action on $M$.

Next,
consider the Tanaka--Webster connection with respect to $\theta$.
Take a local coordinate $(z^{1}, \dots , z^{n})$ of $Y$.
The K\"{a}hler form $\omega$ is written as
\begin{equation*}
	\omega
	= i g_{\alpha \bar{\beta}} d z^{\alpha} \wedge d \overline{z}^{\beta},
\end{equation*}
where $(g_{\alpha \bar{\beta}})$ is a positive definite Hermitian matrix.
An admissible coframe is given by
$(\theta, \theta^{\alpha} := p^{\ast} (d z^{\alpha}),
\theta^{\bar{\alpha}} := p^{\ast} (d \overline{z}^{\alpha}))$.
Since $d \theta = p^{\ast} \omega$,
we have
\begin{equation*}
	d \theta
	= i (p^{\ast} g_{\alpha \bar{\beta}}) \theta^{\alpha} \wedge \theta^{\bar{\beta}},
\end{equation*}
which implies that $h_{\alpha \bar{\beta}} = p^{\ast} g_{\alpha \bar{\beta}}$.
The connection form $\psi_{\alpha}{}^{\beta}$ of the K\"{a}hler metric
with respect to the frame $(\partial / \partial z^{\alpha})$ satisfies
\begin{equation} \label{eq:structure-equation-for-Kahler-metric}
	0 = d (d z^{\beta}) = d z^{\alpha} \wedge \psi_{\alpha}{}^{\beta},
	\qquad
	d g_{\alpha \bar{\beta}}
	= \psi_{\alpha}{}^{\gamma} g_{\gamma \bar{\beta}}
	+ g_{\alpha \bar{\gamma}} \psi_{\bar{\beta}}{}^{\bar{\gamma}}.
\end{equation}
We write as $\Psi_{\alpha}{}^{\beta}$ the curvature form of the K\"{a}hler metric.
Pulling back \cref{eq:structure-equation-for-Kahler-metric} by $p$ gives
\begin{equation*}
	d \theta^{\beta} = \theta^{\alpha} \wedge (p^{\ast} \psi_{\alpha}{}^{\beta}),
	\qquad
	d h_{\alpha \bar{\beta}}
	= (p^{\ast} \psi_{\alpha}{}^{\gamma}) h_{\gamma \bar{\beta}}
	+ h_{\alpha \bar{\gamma}} (p^{\ast} \psi_{\bar{\beta}}{}^{\bar{\gamma}}).
\end{equation*}
This yields
$\omega_{\alpha}{}^{\beta} = p^{\ast} \psi_{\alpha}{}^{\beta}$.  In particular, the pseudohermitian torsion vanishes identically;
that is,
$(M, T^{1, 0}, \theta)$ is a Sasakian manifold.
Moreover,
the curvature form $\Pi_{\alpha}{^{\beta}}$ of the Tanaka--Webster connection
is given by $\Pi_{\alpha}{}^{\beta} = p^{\ast} \Psi_{\alpha}{}^{\beta}$.

\section{Representatives for characteristic classes}
\label{sec:chern}

In this section we give some equivalent representatives for the characteristic classes of a CR manifold.  Given an invariant polynomial $\Phi$ of degree $k$ and a matrix $Y_\alpha{}^\beta$ of two-forms, we define the \emph{characteristic form $c_\Phi(Y_\alpha{}^\beta)$} by
\[ c_\Phi(Y_\alpha{}^\beta) := \Phi_{\alpha_1\dotsm\alpha_k}^{\beta_1\dotsm\beta_k}Y_{\beta_1}{}^{\alpha_1}\dotsm Y_{\beta_k}{}^{\alpha_k} ; \]
throughout this section we multiply forms using the exterior product.  The characteristic class of $(M^{2n+1},T^{1,0})$ determined by $\Phi$ is
\[ c_\Phi(T^{1,0}) := \left[c_\Phi\left(\frac{i}{2\pi}\Pi_\alpha{}^\beta\right) \right] . \]
It is well-known $c_\Phi(T^{1,0})$ is independent of the choice of contact form.

We are interested in two other $\End(T^{1,0})$-valued two-forms on a pseudohermitian manifold $(M^{2n+1},T^{1,0},\theta)$, namely
\begin{align}
 \label{eqn:Omega} \Omega_\alpha{}^\beta & := R_\alpha{}^\beta{}_\mu{}^\nu\theta^\mu\theta_\nu -  \nabla^\beta A_{\alpha\mu}\theta\theta^\mu + \nabla_\alpha A^{\beta\nu}\theta\theta_\nu , \\
 \label{eqn:Xi} \Xi_\alpha{}^\beta & := S_\alpha{}^\beta{}_\mu{}^\nu\theta^\mu\theta_\nu - V_\alpha{}^\beta{}_\mu\theta\theta^\mu + V^\beta{}_\alpha{}^\nu\theta\theta_\nu .
\end{align}
It is known~\cite[Equations~(2.2) and~(2.4)]{Lee1988} that
\begin{equation}
 \label{eqn:Pi-to-Omega}
 \Omega_\alpha{}^\beta = \Pi_\alpha{}^\beta - i\theta_\alpha\tau^\beta + i\tau_\alpha\theta^\beta .
\end{equation}
The main results of this section are that $c_\Phi(\Omega_\alpha{}^\beta)$ is closed and the induced element in $H^{2k}(M;\bR)$ agrees with $[c_\Phi(\Pi_\alpha{}^\beta)]$, and moreover the same is true for $c_\Phi(\Xi_\alpha{}^\beta)$ on pseudo-Einstein manifolds.  This requires three observations.

Our first observation is that $[c_\Phi(\Pi_\alpha{}^\beta)]=[c_\Phi(\Omega_\alpha{}^\beta)]$.

\begin{proposition}
 \label{prop:reduce-chern-Omega}
 Let $(M^{2n+1},T^{1,0},\theta)$ be a pseudohermitian manifold and let $\Phi$ be an invariant polynomial of degree $k$.  Then $c_\Phi(\Omega_\alpha{}^\beta)$ is closed and
 \[ \left[c_\Phi(\Pi_\alpha{}^\beta)\right] = \left[c_\Phi(\Omega_\alpha{}^\beta)\right] . \]
\end{proposition}

\begin{proof}
 Denote
 \[ T_k(\Pi_\alpha{}^\beta) := \tr \Pi^k := \Pi_{\gamma_1}{}^{\gamma_2}\Pi_{\gamma_2}{}^{\gamma_3}\dotsm \Pi_{\gamma_k}{}^{\gamma_1} ; \]
 note that $T_k(\Pi_\alpha{}^\beta)=k!\chern_k(\Pi_\alpha{}^\beta)$ is proportional to the $k$-th Chern character form.  Since $\{T_k\}_{k=1}^\infty$ generates the algebra of invariant polynomials, it suffices to prove the result for all $T_k$.
 
 Denote
 \[ \Theta_\alpha{}^\beta := i\theta_\alpha\tau^\beta - i\tau_\alpha\theta^\beta , \]
 so that $\Pi_\alpha{}^\beta = \Omega_\alpha{}^\beta + \Theta_\alpha{}^\beta$.  Denote $(\Theta^s)_\alpha{}^\beta:=\Theta_\alpha{}^{\gamma_2}\Theta_{\gamma_2}{}^{\gamma_3}\dotsm\Theta_{\gamma_{s}}{}^\beta$.  We compute that
 \begin{align}
  \label{eqn:Theta_odd} (\Theta^{2s+1})_\alpha{}^\beta & = (-i\tau^\gamma\tau_\gamma d\theta)^s\Theta_\alpha{}^\beta, \\
  \label{eqn:Theta_even} (\Theta^{2s})_\alpha{}^\beta & = (-i\tau^\gamma\tau_\gamma d\theta)^{s-1}(\tau^\rho\tau_\rho\theta_\alpha\theta^\beta - i\tau_\alpha\tau^\beta d\theta)
 \end{align}
 for all $s\in\bN$.
 A direct computation using \cref{eqn:Pi} and the definition of $\omega_\alpha{}^\beta$ yields
 \begin{align*}
  d\Pi_\alpha{}^\beta & = \omega_\alpha{}^\gamma\Pi_\gamma{}^\beta - \Pi_\alpha{}^\gamma\omega_\gamma{}^\beta, \\
  \Pi_\alpha{}^\gamma\theta_\gamma & = -d(\theta\tau_\alpha) - \theta\omega_\alpha{}^\gamma\tau_\gamma, \\
  \theta^\gamma\Pi_\gamma{}^\beta & = d(\theta\tau^\beta) + \theta\tau^\gamma\omega_\gamma{}^\beta .
 \end{align*}
 It follows from these equations that
 \begin{align*}
  d\Omega_\alpha{}^\beta & = \omega_\alpha{}^\gamma\Omega_\gamma{}^\beta - \Omega_\alpha{}^\gamma\omega_\gamma{}^\beta + i\theta_\alpha(d\tau^\beta-\tau^\gamma\omega_\gamma{}^\beta) + i(d\tau_\alpha-\omega_\alpha{}^\gamma\tau_\gamma)\theta^\beta, \\
  \Omega_\alpha{}^\gamma\theta_\gamma & = \theta(d\tau_\alpha - \omega_\alpha{}^\gamma\tau_\gamma), \\
  \theta^\gamma\Omega_\gamma{}^\beta & = -\theta(d\tau^\beta-\tau^\gamma\omega_\gamma{}^\beta) .
 \end{align*}
 We deduce that
 \begin{align}
  \label{eqn:dtOmegat} d\left(\tau^\gamma(\Omega^s)_\gamma{}^\beta\tau_\beta\right) & = (d\tau^\gamma - \tau^\rho\omega_\rho{}^\gamma)(\Omega^s)_\gamma{}^\beta\tau_\beta - \tau^\gamma(\Omega^s)_\gamma{}^\beta(d\tau_\beta - \omega_\beta{}^\rho\tau_\rho), \\
  \label{eqn:OmegaTOmega} \Omega_\alpha{}^\beta\Theta_\beta{}^\gamma\Omega_\gamma{}^\rho & = i\theta\left((d\tau_\alpha-\omega_\alpha{}^\beta\tau_\beta)\tau^\gamma\Omega_\gamma{}^\rho - \Omega_\alpha{}^\beta\tau_\beta(d\tau^\rho-\tau^\gamma\omega_\gamma{}^\rho)\right), \\
  \label{eqn:OmegaTTOmega} \Omega_\alpha{}^\beta(\Theta^2)_\beta{}^\gamma\Omega_\gamma{}^\rho & = -i\Omega_\alpha{}^\beta\tau_\beta\tau^\gamma\Omega_\gamma{}^\rho d\theta
 \end{align}
 for all integers $s\geq0$.

 Given $s\in\bN$ and $N\in\bN$, define $\mO_{s,N}\in\Lambda^{2N+4s-2}T^\ast M$ by
 \[ \mO_{s,N} := \sum_{\substack{j_1,\dotsc,j_s\geq1\\j_1+\dotsm+j_s=N}} \Theta_{\gamma_s}{}^{\beta_1}(\Omega^{j_1})_{\beta_1}{}^{\gamma_1}(\Theta^2)_{\gamma_1}{}^{\beta_2}(\Omega^{j_2})_{\beta_2}{}^{\gamma_2} \dotsm (\Theta^2)_{\gamma_{s-1}}{}^{\beta_s}(\Omega^{j_s})_{\beta_s}{}^{\gamma_s} , \]
 with the convention $\mO_{s,N}=0$ if $N<s$.  It follows from \cref{eqn:dtOmegat,eqn:OmegaTOmega,eqn:OmegaTTOmega} that
 \begin{equation}
  \label{eqn:mO}
  \mO_{s,N} = -(-i)^s\theta\,d\theta^{s-1} \sum_{\substack{j_1,\dotsc,j_s\geq1\\j_1+\dotsm+j_s=N}} \Psi^{(j_1)}\dotsm\Psi^{(j_{s-1})} d\Psi^{(-1+j_s)},
 \end{equation}
 where
 \[ \Psi^{(j)} := \tau^\alpha(\Omega^j)_\alpha{}^\beta \tau_\beta . \]
 
 Given $s\in\bN$ and $N\in\bN$, define $\mE_{s,N}\in\Lambda^{2N+4s}T^\ast M$ by
 \[ \mE_{s,N} := \sum_{\substack{j_1,\dotsc,j_s\geq1\\j_1+\dotsm+j_s=N}} (\Theta^2)_{\gamma_s}{}^{\beta_1}(\Omega^{j_1})_{\beta_1}{}^{\gamma_1}\dotsm(\Theta^2)_{\gamma_{s-1}}{}^{\beta_s}(\Omega^{j_s})_{\beta_s}{}^{\gamma_s} , \]
 with the convention $\mE_{s,N}=0$ if $N<s$.  It follows from \cref{eqn:OmegaTTOmega} that
 \begin{equation}
  \label{eqn:mE}
  \mE_{s,N} = -(-i)^sd\theta^s\sum_{\substack{j_1,\dotsc,j_s\geq1\\j_1+\dotsm+j_s=N}} \Psi^{(j_1)}\dotsm\Psi^{(j_s)} .
 \end{equation}
 Combining \cref{eqn:mO,eqn:mE} yields
 \begin{multline}
  \label{eqn:mO-to-mE}
  \mO_{s,N} = -\frac{1}{s}\mE_{s,N-1} - i\theta\, d(\tau^\gamma\tau_\gamma)\mE_{s-1,N-1} \\ + \frac{(-i)^s}{s}d\Biggl(\sum_{\substack{j_1,\dotsc,j_s\geq1\\j_1+\dotsm+j_s=N-1}} \Psi^{(j_1)}\dotsm \Psi^{(j_s)}\theta\,d\theta^{s-1}\Biggr) ,
 \end{multline}
 with the convention $\mE_{0,0}=-1$ and $\mE_{0,N}=0$ if $N\geq1$.

 Now consider $T_k(\Pi_\alpha{}^\beta)=T_k(\Omega_\alpha{}^\beta+\Theta_\alpha{}^\beta)$.  Write
 \[ T_k(\Pi_\alpha{}^\beta) = \sum_{s=0}^k f_s, \]
 where $f_s$ is the term obtained by expanding $T_k(\Omega_\alpha{}^\beta+\Theta_\alpha{}^\beta)$ as a polynomial in $\Omega_\alpha{}^\beta$ and $\Theta_\alpha{}^\beta$, and keeping only those terms which are homogeneous of degree $s$ in $\Theta_\alpha{}^\beta$.
 
 First note that, for $s\geq0$ and $2s+2<k$,
 \begin{equation}
  \label{eqn:f-even}
  f_{2s+2} = k\sum_{j=1}^{s+1} \frac{1}{j}\binom{s}{j-1}(-i\tau^\gamma\tau_\gamma d\theta)^{s+1-j}\mE_{j,k-2-2s} .
 \end{equation}
 To obtain this formula, first note that \cref{eqn:OmegaTOmega,eqn:Theta_odd} imply that all products with at least two factors of odd powers $(\Theta^{2\ell+1})_\alpha{}^\beta$, $\ell\geq0$, of $\Theta_\alpha{}^\beta$ which are separated by powers of $\Omega$ must vanish; e.g.\ $\Omega_\alpha{}^\beta\Theta_\beta{}^\gamma\Omega_\gamma{}^\rho\Theta_\rho{}^\alpha=0$.  Therefore $f_{2s+2}$ can be written as a polynomial in $(\Theta^2)_\alpha{}^\beta$ and $\Omega_\alpha{}^\beta$.  Group the summands according to how many times a positive power of $(\Theta^2)_\alpha{}^\beta$ is multiplied on the left and the right by $\Omega_\alpha{}^\beta$.  Using \cref{eqn:Theta_even}, we see that the sum of all possible terms where this happens $j$ times is a multiple $c_j$ of
 \[ (-i\tau^\gamma\tau_\gamma d\theta)^{s+1-j}\mE_{j,k-2-2s} . \]
 To compute the multiple, note that in the definition of $\mE_{j,k-2-2s}$, there are $j$ positions --- corresponding to each of the factors of $(\Theta^2)_\alpha{}^\beta$ --- where the extra $s+1-j$ copies of $(\Theta^2)_\alpha{}^\beta$ can be multiplied.  There are $\binom{s}{j-1}$ ways these products appear in the expansion of $T_k(\Pi_\alpha{}^\beta)$.  Since $\mE_{j,k-2-2s}$ is symmetric in the ordering of the factors of $(\Theta^2)_{\alpha}{}^\beta$ and there are $k$ different ways to cyclically permute the terms of $\mE_{j,k-2-2s}$, we conclude that $c_j=\frac{k}{j}\binom{s}{j-1}$.  This yields \cref{eqn:f-even}.
 
 \Cref{eqn:Theta_even} implies that if $k$ is even, then
 \[ f_k = -2(-i\tau^\gamma\tau_\gamma d\theta)^{k/2} . \]
 Combining \cref{eqn:f-even} and our conventions that $\mE_{0,0}=-1$ and $\mE_{j,0}=0$ if $j\geq1$ implies that
 \begin{equation}
  \label{eqn:f-even-0}
  f_{2s+2} = k\sum_{j=0}^{s+1}\frac{1}{j}\binom{s}{j-1}(-i\tau^\gamma\tau_\gamma d\theta)^{s+1-j}\mE_{j,k-2-2s}
 \end{equation}
 for all $s\geq0$, where we recall that $\frac{1}{j}\binom{s}{j-1}=\frac{1}{s+1-j}\binom{s}{j}$ to make sense of the coefficient when $j=0$.
 
 Second note that, for $s\geq0$,
 \begin{equation}
  \label{eqn:f-odd}
  f_{2s+1} = k\sum_{j=1}^{s+1} \binom{s}{j-1}(-i\tau^\gamma\tau_\gamma d\theta)^{s+1-j}\mO_{j,k-1-2s} .
 \end{equation}
 We obtain this formula by following the same procedure as above, except that now there must be a single factor of an odd power of $\Theta_\alpha{}^\beta$, and the location of this factor specifies a preferred ordering of the terms of the expansion, up to cyclic permutation.
 
 Finally, it follows from \cref{eqn:mO-to-mE,eqn:f-odd,eqn:f-even-0} that
 \begin{align*}
  f_{2s+1}+f_{2s+2} & = k\sum_{j=1}^{s+1} \binom{s}{j-1}(-i\tau^\gamma\tau_\gamma d\theta)^{s+1-j}\Biggl[ -i\theta\, d(\tau^\gamma\tau_\gamma)\mE_{j-1,k-2-2s} \\
   & \qquad\qquad + \frac{(-i)^j}{j}d\biggl(\sum_{\substack{\ell_1,\dotsc,\ell_j\geq1\\\ell_1+\dotsm+\ell_j=k-2-2s}} \Psi^{(\ell_1)}\dotsm \Psi^{(\ell_j)} \theta\, d\theta^{j-1}\biggr) \Biggr] \\
   & = k\,d\Biggl[ \sum_{j=1}^{s+1} \frac{i}{j}\binom{s}{j-1}\theta\tau^\gamma\tau_\gamma(-i\tau^\gamma\tau_\gamma d\theta)^{s-j}\mE_{j,k-2-2s} \Biggr] .
 \end{align*}
 In particular, $f_{2s+1}+f_{2s+2}$ is exact for all integers $s\geq0$.  Adopting the convention that $f_\ell=0$ for all $\ell\geq k+1$, we may write
 \[ T_k(\Pi_\alpha{}^\beta) = f_0 + \sum_{s=0}^\infty (f_{2s+1}+f_{2s+2}) . \]
 Since $f_0=T_k(\Omega_\alpha{}^\beta)$ and $f_{2s+1}+f_{2s+2}$ is exact, we conclude that $T_k(\Omega_\alpha{}^\beta)$ is closed and $[T_k(\Pi_\alpha{}^\beta)]=[T_k(\Omega_\alpha{}^\beta)]$.
\end{proof}

Our second observation is that the form $c_\Phi(\Xi_\alpha{}^\beta)$ is always closed.

\begin{lemma}
 \label{lem:Xi-closed}
 Let $(M^{2n+1},T^{1,0},\theta)$ be a pseudohermitian manifold and let $\Phi$ be an invariant polynomial of degree $k$.  Then $c_\Phi(\Xi_\alpha{}^\beta)$ is closed.
\end{lemma}

\begin{proof}
 It follows from \cref{lem:bianchi} that
 \begin{multline*}
  d\Xi_\alpha{}^\beta = -iV_\alpha{}^\nu{}_\rho\theta^\beta\theta^\rho\theta_\nu + iV^\beta{}_\mu{}^\sigma\theta_\alpha\theta^\mu\theta_\sigma - i\left(S_\mu{}^\nu{}_\alpha{}^\rho P_\rho{}^\beta - S_\mu{}^\nu{}_\rho{}^\beta P_\alpha{}^\rho\right)\theta\theta^\mu\theta_\nu \\ + iU_\alpha{}^\nu\theta\theta^\beta\theta_\nu - iU_\mu{}^\beta\theta\theta^\mu\theta_\alpha - iQ_{\alpha\gamma}\theta\theta^\beta\theta^\gamma - iQ^{\beta\gamma}\theta\theta_\alpha\theta_\gamma .
 \end{multline*}
 Using the facts that $S_{\alpha\bar\beta\gamma\bar\sigma}$, $V_{\alpha\bar\beta\gamma}$, and $Q_{\alpha\gamma}$ are all symmetric~\cite[Section~2.3]{CaseGover2013}, we readily verify from the above display that $dc_\Phi(\Xi_\alpha{}^\beta)=0$.
\end{proof}

Our third observation is that if $(M^{2n+1},T^{1,0},\theta)$ is pseudo-Einstein, then the cohomology classes $[c_\Phi(\Omega_\alpha{}^\beta)]$ and $[c_\Phi(\Xi_\alpha{}^\beta)]$ agree.

\begin{proposition}
 \label{prop:reduce-chern-Xi}
 Let $(M^{2n+1},T^{1,0},\theta)$ be a pseudo-Einstein manifold and let $\Phi$ be an invariant polynomial of degree $k$.  Then
 \[ \left[c_\Phi(\Omega_\alpha{}^\beta)\right] = \left[c_\Phi(\Xi_\alpha{}^\beta)\right] . \]
\end{proposition}

\begin{proof}
 Since $(M^{2n+1},T^{1,0},\theta)$ is pseudo-Einstein,
 \begin{align*}
  R_\alpha{}^\beta{}_\mu{}^\nu & = S_\alpha{}^\beta{}_\mu{}^\nu + \frac{2}{n}P\delta_\alpha^\beta\delta_\mu^\nu + \frac{2}{n}P\delta_\alpha^\nu\delta^\beta_\mu , \\
  \nabla^\beta A_{\alpha\mu} & = V_\alpha{}^\beta{}_\mu - \frac{2i}{n}\delta_\alpha^\beta\nabla_\mu P - \frac{2i}{n}\delta_\mu^\beta\nabla_\alpha P .
 \end{align*}
 It follows that
 \begin{equation}
  \label{eqn:Omega-simplification}
  \Omega_\alpha{}^\beta = \Xi_\alpha{}^\beta - \frac{2i}{n}\delta_\alpha^\beta d(P\theta) + \frac{2}{n}\left( P\theta^\beta\theta_\alpha + i\nabla_\alpha P\,\theta\theta^\beta + i\nabla^\beta P\,\theta\theta_\alpha\right) .
 \end{equation}

 On the one hand, since $\Phi$ is an invariant polynomial of degree $k$, its trace $\Phi_{\gamma\alpha_2\dotsm\alpha_k}^{\gamma\beta_2\dotsm\beta_k}$ is an invariant polynomial of degree $k-1$.  Also, by \cref{prop:reduce-chern-Omega}, $c_\Phi(\Omega_\alpha{}^\beta)$ is closed.  It follows immediately that $c_\Phi\bigl(\Omega_\alpha{}^\beta+\frac{2i}{n}\delta_\alpha^\beta d(P\theta)\bigr)$ is closed and
 \begin{equation}
  \label{eqn:reduce-Xi1}
  \left[ c_\Phi\bigl( \Omega_\alpha{}^\beta + \frac{2i}{n}\delta_\alpha^\beta d(P\theta) \bigr) \right] = \left[ c_\Phi\bigl(\Omega_\alpha{}^\beta\bigr) \right] .
 \end{equation}
 On the other hand, set
 \[ \Gamma_\alpha{}^\beta := P\theta^\beta\theta_\alpha + i\nabla_\alpha P\,\theta\theta^\beta + i\nabla^\beta P\,\theta\theta_\alpha . \]
 Note that $\Omega_\alpha{}^\beta+\frac{2i}{n}\delta_\alpha^\beta d(P\theta)=\Xi_\alpha{}^\beta + \frac{2}{n}\Gamma_\alpha{}^\beta$.  A straightforward induction argument yields
 \begin{align*}
  \left(\Gamma^m\right)_\alpha{}^\beta & := \Gamma_\alpha{}^{\gamma_2}\Gamma_{\gamma_2}{}^{\gamma_3}\dotsm\Gamma_{\gamma_m}{}^\beta \\
  & = (-1)^{m+1}iP^{m-1}\left(\Gamma_\alpha{}^\beta d\theta^{m-1} + (m-1)(dP)\theta\theta^\beta\theta_\alpha d\theta^{m-2}\right)
 \end{align*}
 for all $m\in\bN$.  In particular, we deduce that
 \begin{align*}
  \left(\Gamma^m\right)_\alpha{}^\alpha & = (-1)^{m+1}d\left(P^m\theta(d\theta)^{m-1}\right), \\
  \Xi_\alpha{}^\beta\left(\Gamma^m\right)_\beta{}^\alpha & = 0, \\
  \Xi_\alpha{}^\gamma\left(\Gamma^m\right)_\gamma{}^\rho\Xi_\rho{}^\beta & = 0
 \end{align*}
 for all $m\in\bN$.  Combining this with \cref{lem:Xi-closed} yields
 \begin{equation}
  \label{eqn:reduce-Xi2}
  \left[ c_\Phi\bigl( \Xi_\alpha{}^\beta + \frac{2}{n}\Gamma_\alpha{}^\beta \bigr) \right] = \left[ c_\Phi\bigl( \Xi_\alpha{}^\beta \bigr) \right] .
 \end{equation}

 The conclusion follows immediately from \cref{eqn:Omega-simplification,eqn:reduce-Xi1,eqn:reduce-Xi2}.
\end{proof}

\begin{remark}
 In his proof of~\cite[Proposition 5.4]{Marugame2019}, Marugame showed that the conclusion of \cref{prop:reduce-chern-Xi} is true if the assumption on $(M^{2n+1},T^{1,0},\theta)$ is weakened to only assume that $c_1(T^{1,0})=0$ in $H^2(M;\bR)$.   
\end{remark}

\section{The invariance of $X_\alpha^\Phi$ and the total $\mI_\Phi^\prime$-curvature}
\label{sec:invariance}

In this section we prove that $X_\alpha^\Phi$ and $\nabla^\alpha X_\alpha^\Phi$ are CR invariant, derive the transformation formula for $\mI_\Phi^\prime$ under change of contact form, and conclude that the total $\mI_\Phi^\prime$-curvature is a secondary CR invariant.

First we prove that $X_\alpha^\Phi$ and $\nabla^\alpha X_\alpha^\Phi$ are CR invariant.

\begin{proof}[Proof of~\cref{thm:X-invariant}]
 On the one hand, since $c_\Phi(S)$ is a scalar CR invariant of weight $-n$, we conclude from~\cref{lem:variations} that
 \begin{equation}
  \label{eqn:Xinv-grad}
  D_\theta\nabla_\alpha c_\Phi(S)(\Upsilon) = -nc_\Phi(S)\Upsilon_\alpha .
 \end{equation}
 On the other hand, since $D_\theta$ is a derivation and $\mS^\Phi$ is a local CR invariant, we conclude from~\cref{lem:variations} that
 \begin{equation}
  \label{eqn:Xinv-div}
  D_\theta\left(i(\mS^\Phi)_\alpha{}^\beta{}_\mu{}^\nu V_\beta{}^\mu{}_\nu\right)(\Upsilon) = -(\mS^\Phi)_\alpha{}^\beta{}_\mu{}^\nu S_\beta{}^\rho{}_\nu{}^\mu \Upsilon_\rho .
 \end{equation}
 Since $M$ has CR dimension $n$, it holds that
 \[ (\mU^\Phi)_\alpha{}^\beta := \delta_{\alpha\alpha_1\dotsm\alpha_n}^{\beta\beta_1\dotsm\beta_n}\Phi_{\mu_1\dotsm\mu_n}^{\nu_1\dotsm\nu_n}S_{\beta_1}{}^{\alpha_1}{}_{\nu_1}{}^{\mu_1}\dotsm S_{\beta_n}{}^{\alpha_n}{}_{\nu_n}{}^{\mu_n} = 0 . \]
 In particular,
 \begin{equation}
  \label{eqn:Xinv-trivial}
  0 = (\mU^\Phi)_\alpha{}^\beta\Upsilon_\beta = c_\Phi(S)\Upsilon_\alpha - n(\mS^\Phi)_\alpha{}^\beta{}_\mu{}^\nu S_\beta{}^\rho{}_\nu{}^\mu\Upsilon_\rho .
 \end{equation}
 Combining~\cref{eqn:Xinv-grad,eqn:Xinv-div,eqn:Xinv-trivial} implies that $X_\alpha^\Phi$ is a CR invariant $(1,0)$-form of weight $-n$.  It follows immediately from \cref{lem:variations} that $\nabla^\alpha X_\alpha^\Phi$ is a CR invariant of weight $-n-1$.
\end{proof}

Next we derive the transformation formula for $\mI_\Phi^\prime$ under change of contact form.

\begin{proof}[Proof of~\cref{thm:mI-transformation}]
 It follows from~\cref{lem:variations} that
 \begin{align}
  \label{eqn:mI-lapl} D_\theta\left(\Delta_bc_\Phi(S)-2nPc_\Phi(S)\right)(\Upsilon) & = -2n\Real \Upsilon^\gamma\nabla_\gamma c_\Phi(S) , \\
  \label{eqn:mI-V} D_\theta\left(V_\beta{}^{\mu_1}{}_{\nu_1} V^\alpha{}_{\nu_2}{}^{\mu_2}\right)(\Upsilon) & = -2\Real iV_\beta{}^{\mu_1}{}_{\nu_1}S_\rho{}^\alpha{}_{\nu_2}{}^{\mu_2}\Upsilon^\rho .
 \end{align}
 Since $\mT^\Phi$ is a local CR invariant, we conclude from~\cref{lem:variations,eqn:mI-V} that
 \begin{equation}
  \label{eqn:mI-mT}
  D_\theta\mV(\Upsilon) = 2\Real i(S^\Phi)_\alpha{}^\beta{}_\mu{}^\nu V_\beta{}^\mu{}_\nu \Upsilon^\alpha ,
 \end{equation}
 where
 \[ \mV := (\mT^\Phi)_\alpha{}^\beta{}_{\mu_1}{}^{\nu_1}{}_{\mu_2}{}^{\nu_2}\left((n-1)V_\beta{}^{\mu_1}{}_{\nu_1} V^\alpha{}_{\nu_2}{}^{\mu_2} - S_\beta{}^\alpha{}_{\nu_1}{}^{\mu_1} U_{\nu_2}{}^{\mu_2}\right) . \]
 Combining~\cref{eqn:mI-lapl,eqn:mI-mT,thm:X-invariant} yields
 \[ \frac{\partial}{\partial t} e^{(n+1)t\Upsilon} \bigl(\mI_\Phi^\prime\bigr)^{e^{t\Upsilon}\theta} = e^{(n+1)t\Upsilon}\bigl(2\Real  X_\alpha^\Phi \Upsilon^\alpha\bigr)^{e^{t\Upsilon}\theta} = 2\Real(X_\alpha^\Phi\Upsilon^\alpha)^\theta . \]
 Integrating this equation in $t\in[0,1]$ yields the desired result.
\end{proof}

The rest of this section is devoted to the proof that the total $\mI_\Phi^\prime$-curvature is a secondary CR invariant.  The main task is to relate $X_\alpha^\Phi$ to the characteristic class $c_\Phi(T^{1,0})$.

\begin{proposition}
 \label{prop:chern}
 Let $(M^{2n+1},T^{1,0},\theta)$ be a pseudo-Einstein manifold and let $\Phi$ be an invariant polynomial of degree $n$.  Set
 \[ \xi^\Phi := X_\alpha^\Phi\theta\wedge\theta^\alpha\wedge d\theta^{n-1} + X_{\bar\beta}^\Phi\theta\wedge\theta^{\bar\beta}\wedge d\theta^{n-1} . \]
 Then $\xi^\Phi$ is closed.  Moreover, $n[\xi^\Phi]=-(2\pi)^n(n-1)!c_\Phi(T^{1,0})$ in $H^{2n}(M;\bR)$.
\end{proposition}

\begin{proof}
 Combining \cref{prop:reduce-chern-Omega,prop:reduce-chern-Xi} yields
 \begin{equation}
  \label{eqn:chern}
  (2\pi)^nc_\Phi\left(T^{1,0}\right) = \left[c_\Phi\left(i\Xi_\alpha{}^\beta\right)\right] ,
 \end{equation}
 where $\Xi_\alpha{}^\beta$ is defined by \cref{eqn:Xi}.  An easy computation yields
 \begin{multline*}
  c_\Phi\left(i\Xi_\alpha{}^\beta\right) = i^n\Phi_{\alpha_1\dotsm\alpha_n}^{\beta_1\dotsm\beta_n}S_{\beta_1}{}^{\alpha_1}{}_{\mu_1}{}^{\nu_1}\dotsm S_{\beta_n}{}^{\alpha_n}{}_{\mu_n}{}^{\nu_n}\theta^{\mu_1}\theta_{\nu_1}\dotsm\theta^{\mu_n}\theta_{\nu_n} \\ 
   - ni^n\Phi_{\alpha_1\dotsm\alpha_n}^{\beta_1\dotsm\beta_n}V_{\beta_1}{}^{\alpha_1}{}_{\mu_1}S_{\beta_2}{}^{\alpha_2}{}_{\mu_2}{}^{\nu_2}\dotsm S_{\beta_n}{}^{\alpha_n}{}_{\mu_n}{}^{\nu_n} \theta\theta^{\mu_1}\theta^{\mu_2}\theta_{\nu_2}\dotsm\theta^{\mu_n}\theta_{\nu_n} \\
   + ni^n\Phi_{\alpha_1\dotsm\alpha_n}^{\beta_1\dotsm\beta_n}V^{\alpha_1}{}_{\beta_1}{}^{\nu_1}S_{\beta_2}{}^{\alpha_2}{}_{\mu_2}{}^{\nu_2}\dotsm S_{\beta_n}{}^{\alpha_n}{}_{\mu_n}{}^{\nu_n}\theta\theta_{\nu_1}\theta^{\mu_2}\theta_{\nu_2}\dotsm\theta^{\mu_n}\theta_{\nu_n} .
 \end{multline*}
 Since $\dim T^{1,0}=n$, it must hold that $c_\Phi(\Xi_\alpha{}^\beta)$ is in the span of $d\theta^n$, $\theta\theta^\alpha d\theta^{n-1}$, and $\theta\theta_\beta d\theta^{n-1}$.  We then compute that
 \begin{multline*}
  c_\Phi\left(i\Xi_\alpha{}^\beta\right) = \frac{1}{n!}c_\Phi(S)\,d\theta^n - \frac{ni}{(n-1)!}(\mS^\Phi)_\alpha{}^\beta{}_\mu{}^\nu V_\beta{}^\mu{}_\nu \theta\theta^\alpha d\theta^{n-1} \\ + \frac{ni}{(n-1)!}(\mS^\Phi)_\alpha{}^\beta{}_\mu{}^\nu V^\alpha{}_\nu{}^\mu\theta\theta_\beta d\theta^{n-1} .
 \end{multline*}
 In particular,
 \[ c_\Phi\left(i\Xi_\alpha{}^\beta\right) = \frac{1}{n!}d\left(c_\Phi(S)\,\theta\,d\theta^{n-1}\right) - \frac{n}{(n-1)!}\xi^\Phi . \]
 We conclude that $\xi^\Phi$ is closed and $\bigl[c_\Phi(i\Xi_\alpha{}^\beta)\bigr]=-\frac{n}{(n-1)!}[\xi^\Phi]$.  The conclusion now follows from \cref{eqn:chern}.
\end{proof}

We now conclude that the total $\mI_\Phi^\prime$-curvature is a secondary invariant.

\begin{proof}[Proof of~\cref{thm:secondary_invariant}]
 Let $\xi^\Phi$ be as in~\cref{prop:chern}.  We may thus consider the cohomology class $[\xi^\Phi]\in H^{2n}(M;\bR)$.  Recall~\cite[Lemma~3.1]{Lee1988} that $\Upsilon\in C^\infty(M)$ is CR pluriharmonic if and only if $d_b^c\Upsilon:=i(\Upsilon_{\bar\beta}\theta^{\bar\beta}-\Upsilon_\alpha\theta^\alpha)\in\Gamma\left(T^\ast M/\lp\theta\rp\right)$ is closed in the sense of Rumin~\cite{Rumin1994}.  In particular, for any $\Upsilon\in \mP$, the cup product $[\xi^\Phi]\cup[d_b^c\Upsilon]:=[\xi^\Phi\wedge d_b^c\Upsilon]$ is well-defined in $H^{2n+1}(M;\bR)$.  A straightforward computation implies that
 \begin{equation}
  \label{eqn:cup}
  \left\lp [\xi^\Phi] \cup [d_b^c\Upsilon], [M] \right\rp = \frac{2}{n}\Real\int_M X_\alpha^\Phi\Upsilon^\alpha\,\theta\wedge d\theta^n ,
 \end{equation}
 where $[M]$ is the fundamental class of $M$.
 
 Next, a result of Takeuchi~\cite[Theorem~1.1]{Takeuchi2020} implies that $c_\Phi(T^{1,0})=0$.  Combining \cref{prop:chern,eqn:cup} yields $\Real\int_M X_\alpha^\Phi\Upsilon^\alpha\,\theta\wedge d\theta^n=0$ for all $\Upsilon\in\mP$.  Combining this with~\cref{eqn:mI-transformation} yields the desired result.
\end{proof}

\section{Counterexamples to the Hirachi conjecture}
\label{sec:hirachi}

As noted in the introduction, if \cref{conj:strong_hirachi} holds, then
the $\mI_{\Phi}^{\prime}$-curvature must be a linear combination of a local CR invariant and a divergence.
However,
there is a local CR invariant whose vanishing is necessary for $\mI_\Phi^\prime$ to be a linear combination of a local CR invariant and a divergence.

\begin{lemma}
\label{lem:construction-of-counterexample}
	Let $\Phi$ be an invariant polynomial of degree $n$.
	If there exists a pseudohermitian manifold $(M^{2n+1}, T^{1, 0}, \theta)$
	such that $\Real \nabla^{\alpha} X^{\Phi}_{\alpha}$ is not identically zero,
	then $\mathcal{I}_{\Phi}^{\prime}$ is not the linear combination of a local CR invariant and a divergence.
\end{lemma}

\begin{proof}
    Suppose to the contrary that $\mathcal{I}_{\Phi}^{\prime}$ is a linear combination of a local CR invariant and a divergence.
    Then the total integral of $\mathcal{I}_{\Phi}^{\prime}$ is independent of the choice of contact form.
    However, under the conformal change
    \begin{equation*}
        \widehat{\theta} = \exp(\Real \nabla^{\alpha} X^{\Phi}_{\alpha}) \cdot \theta,
    \end{equation*}
    \cref{thm:mI-transformation} implies that
    \begin{equation*}
        \int_{M} \widehat{\mathcal{I}}_{\Phi}^{\prime} \widehat{\theta} \wedge d \widehat{\theta}^{n}
        = \int_{M} \mathcal{I}_{\Phi}^{\prime} \theta \wedge d \theta^{n}
            - 2 \int_{M} (\Real \nabla^{\alpha} X^{\Phi}_{\alpha})^{2}  \theta \wedge d \theta^{n}.
    \end{equation*}
    Since $\Real \nabla^{\alpha} X^{\Phi}_{\alpha} \neq 0$, the total integral of $\mI_\Phi^\prime$ depends on $\theta$, a contradiction.
\end{proof}

Given an invariant polynomial $\Phi$ of degree $n$, \cref{lem:construction-of-counterexample} and the discussion of the introduction shows that one need only find an example of a pseudo-Einstein manifold $(M^{2n+1},T^{1,0},\theta)$ for which $\Real\nabla^\alpha X_\alpha^\Phi$ is not identically zero in order to conclude that that $\mI_\Phi^\prime$ gives a counterexample to \cref{conj:strong_hirachi}.  Indeed, since $\Real\nabla^\alpha X_\alpha^\Phi$ is CR invariant, it suffices to find a pseudohermitian manifold $(M^{2n+1},T^{1,0},\theta)$ which admits a pseudo-Einstein contact form and is such that $\Real\nabla^\alpha X_\alpha^\Phi$ is not identically zero.  We shall present two ways to find such a manifold.

First, in \cref{sec:perturb}, we compute the change of $\Real\nabla^\alpha X_\alpha^\Phi$ along a particular perturbation of the round CR $(2n+1)$-sphere.  This approach is computationally challenging and can be regarded as a local (in the space of CR structures on $S^{2n+1}$) generalization of computations of Reiter and Son~\cite{ReiterSon2019} for five-dimensional real ellipsoids.

Second, in \cref{sec:calabi-yau}, we compute $\Real\nabla^\alpha X_\alpha^\Phi$ on circle bundles over a Calabi--Yau manifold in the case when $\Phi$ is the generalized Kronecker delta.  This approach is computationally simple and relies on explicit examples of degenerating sequences of Calabi--Yau manifolds in complex dimensions two and three.

\section{Counterexample via perturbations of $S^{2n+1}$}
\label{sec:perturb}

The purpose of this section is to prove \cref{thm:counterexample-via-perturb-intro} by considering the $\mI_\Phi^\prime$-curvatures on perturbations of the round CR $(2n+1)$-sphere.  To that end, we need to know the first variation of the Chern tensor $S_{\alpha\bar\beta\gamma\bar\sigma}$ along a suitable deformation.  This formula is known~\cite{Hirachi2014}, but since we cannot find a proof in the literature, we provide one here.

\begin{lemma}
 \label{lem:linearize_chern}
 Let $\rho_t\colon\bC^{n+1}\to\bR$ be a one-parameter family of smooth functions such that $\rho_0(z)=1-\lv z\rv^2$.  Set $M_t:=\rho_t^{-1}(0)$, $T_t^{1,0}:=T^{1,0}\bC^{n+1}\cap(TM_t\otimes\bC)$, and $\theta_t:=\Imaginary\dbbar\rho_t\rv_{M_t}$.  Let $F_t\colon\bC^{n+1}\to\bC^{n+1}$ be a one-parameter family of diffeomorphisms such that $F_0=\Id$, $\rho_t\circ F_t=\rho_0$, and $F_t^\ast\ker\theta_t=\ker\theta_0$.  Denote by $S^t:=F_t^\ast S^{\theta_t}$ the pullback of the Chern tensor of $\theta_t$ by $F_t$.  Then
 \begin{equation}
  \label{eqn:linearize_chern}
  \dot S_{\alpha\bar\beta\gamma\bar\sigma} := \biggl(\left. \frac{\partial}{\partial t}\right|_{t=0} S^t\biggr)_{\alpha\bar\beta\gamma\bar\sigma}= \tf(\dot\rho_t)_{\alpha\bar\beta\gamma\bar\sigma} ,
 \end{equation}
 where $\tf u_{\alpha\bar\beta\gamma\bar\sigma}$ denotes the totally trace-free part of $u_{\alpha\bar\beta\gamma\bar\sigma}$,
 \begin{multline*}
  \tf u_{\alpha\bar\beta\gamma\bar\sigma} := u_{\alpha\bar\beta\gamma\bar\sigma} - \frac{1}{n+2}\left(u_{\alpha\bar\beta\mu}{}^\mu h_{\gamma\bar\sigma} + u_{\gamma\bar\beta\mu}{}^\mu h_{\alpha\bar\sigma} + u_{\alpha\bar\sigma\mu}{}^\mu h_{\gamma\bar\beta} + u_{\gamma\bar\sigma\mu}{}^\mu h_{\alpha\bar\beta}\right) \\ + \frac{1}{(n+1)(n+2)}u_\mu{}^\mu{}_\nu{}^\nu\left(h_{\alpha\bar\beta}h_{\gamma\bar\sigma} + h_{\alpha\bar\sigma}h_{\gamma\bar\beta}\right) ,
 \end{multline*}
 and $h_{\alpha\bar\beta}$ is the Levi form of the round CR $(2n+1)$-sphere $(M_0,T_0^{1,0})$.
\end{lemma}

\begin{remark}
 The existence of diffeomorphisms $F_t\colon\bC^{n+1}\to\bC^{n+1}$ as in the statement of \cref{lem:linearize_chern} is guaranteed by~\cite[Lemma~4.1]{HirachiMarugameMatsumoto2015}.  Note that the restriction $F_t\colon M_0\to M_t$ is a contact diffeomorphism.
\end{remark}

\begin{proof}
 Fix $p\in M_0$.  By permuting coordinates if necessary, we may assume that $(w,z)\in\bC\times\bC^n=\bC^{n+1}$ are such that $(\rho_t)_w:=\frac{\partial\rho_t}{\partial w}$ is nowhere zero in a neighborhood of $(0,p)$ in $\bR\times\bC^{n+1}$.  Consider the frame $Z_\alpha^t:=\partial_{z^\alpha}-\frac{(\rho_t)_\alpha}{(\rho_t)_w}\partial_w$ of $T_t^{1,0}$ near $F_t(p)$.  Applying~\cite[Theorem~3.1]{ReiterSon2019} yields 
 \begin{multline}
  \label{eqn:reiter_son_formula}
  (S^{\theta_t})_{\alpha\bar\beta\gamma\bar\sigma} = \tf \biggl( \mR_{\alpha\bar\beta\gamma\bar\sigma}(\rho_t) + h^{j\bar k}D_{\alpha\gamma}(\rho_{\bar k})D_{\bar\beta\bar\sigma}(\rho_j) \\ + h_{\bar\beta\bar\sigma}\xi^{\bar k}D_{\alpha\gamma}(\rho_{\bar k}) + h_{\alpha\gamma}\xi^jD_{\bar\beta\bar\sigma}(\rho_j) - \lv\xi\rv^2 h_{\alpha\gamma}h_{\bar\beta\bar\sigma}\biggr) ,
 \end{multline}
 where $\xi=\xi(t)$ is the unique $(1,0)$-vector field in $\bC^{n+1}$ such that $\partial\rho_t(\xi)=1$ and $\xi\contr i\partial\overline{\partial}\rho_t\equiv0\mod\overline{\partial}\rho_t$,
 \begin{align*}
  \mR_{\alpha\bar\beta\gamma\bar\sigma}(\phi) & := \phi_{Z\bar ZZ\bar Z}(Z_\alpha,Z_{\bar\beta},Z_\gamma,Z_{\bar\sigma}), \\
  h_{\alpha\bar\beta} & := D_{\alpha\bar\beta}(\rho_t), \\
  D_{\alpha\bar\beta}(\phi) & := \phi_{Z\bar Z}(Z_\alpha,Z_{\bar\beta}), \\
  D_{\alpha\gamma}(\phi) & := \phi_{ZZ}(Z_\alpha,Z_\gamma) ,
 \end{align*}
 for all complex-valued smooth functions $\phi$ on $\bC^{n+1}$, and $\phi_{Z\bar Z}$ denotes the $(1,1)$-part of the Hessian of $\phi$, with similar interpretation of $\phi_{ZZ}$ and $\phi_{Z\bar ZZ\bar Z}$.  (The apparent change of sign from~\cite{ReiterSon2019} is because we take our defining function to be positive in the domain bounded by $\rho_t^{-1}(0)$.)  We emphasize that we regard \cref{eqn:reiter_son_formula} as defining a set of smooth functions determined by the frame $\{Z^t_\alpha\}$ of $T_t^{1,0}$ and its conjugate.  By definition,
 \[ S^t = \bigl( (S^{\theta_t})_{\alpha\bar\beta\gamma\bar\sigma}\circ F_t\bigr) \bigl((F_t^\ast\theta_t^\alpha)\wedge(F_t^\ast\theta_t^{\bar\beta})\bigr)\otimes \bigl((F_t^\ast\theta_t^\gamma)\wedge(F_t^\ast\theta_t^{\bar\sigma})\bigr) , \]
 where $\{\theta_t^\alpha\}$ is the admissible coframe of $T_t^{\ast(1,0)}$ dual to $\{Z_\alpha^t\}$.  Since $S^{\theta_0}=0$, we see that
 \[ \left.\frac{\partial}{\partial t}\right|_{t=0}S^t = \left.\frac{\partial}{\partial t}\right|_{t=0} \bigl((S^{\theta_t})_{\alpha\bar\beta\gamma\bar\sigma}\circ F_t\bigr)\,(\theta_0^\alpha\wedge\theta_0^{\bar\beta})\otimes(\theta_0^\gamma\otimes\theta_0^{\bar\sigma}) . \]
 Combining this with \cref{eqn:reiter_son_formula} and the facts $h_{\alpha\gamma}=h_{\bar\beta\bar\sigma}=0$ and $D_{\alpha\gamma}(\rho_{\bar k}) = D_{\bar\beta\bar\sigma}(\rho_j)=0$ at $t=0$ yields \cref{eqn:linearize_chern}.
\end{proof}

We prove \cref{thm:counterexample-via-perturb-intro} by applying \cref{lem:linearize_chern} to the specific family
\begin{equation}
 \label{eqn:ellipsoid_deformation}
 \rho_t := 1 - \lv z\rv^2 - \lv w\rv^2 + \frac{t}{4}\lv w\rv^4
\end{equation}
of defining functions, where $(w,z)\in\bC\times\bC^{n}$.  In fact, we prove the following sharper result:

\begin{theorem}
 \label{thm:hirachi_ellipsoid}
 Let $\varsigma=(\varsigma_1,\varsigma_2,\varsigma_3,\dotsc,\varsigma_n)\in\bN^n$ be such that $\varsigma_1=0$ and $n=\sum_{k=1}^n k\varsigma_k$.  For $t$ sufficiently close to zero, consider the pseudohermitian manifolds $(M_t,T_t^{1,0},\theta_t)$ and contact diffeomorphisms $F_t\colon M_0\to M_t$ as in \cref{lem:linearize_chern}.  Let $\Phi=\Phi(\varsigma)$ be as in \cref{eqn:Phi_varsigma}.  Then
 \begin{equation}
  \label{eqn:hirachi_ellipsoid_subcritical_derivative}
  \left. \frac{\partial^k}{\partial t^k}\right|_{t=0} \Real F_t^\ast \left(\nabla^\alpha X_\alpha^\Phi\right)^{\theta_t} = 0
 \end{equation}
 for all nonnegative integers $k<n$ and
 \begin{equation}
  \label{eqn:hirachi_ellipsoid_derivative}
  \left. \frac{\partial^n}{\partial t^n}\right|_{t=0} \Real F_t^\ast\left(\nabla^\alpha X_\alpha^\Phi\right)^{\theta_t} \not= 0 .
 \end{equation}
 In particular, for all $t\not=0$ sufficiently close to zero, it holds that $\Real (\nabla^\alpha X_\alpha^\Phi)^{\theta_t} \not= 0$.
\end{theorem}

The proof of \cref{thm:hirachi_ellipsoid} only requires that $\rho_0$ is the defining function of the round $(2n+1)$-sphere and the formula for $\left.\frac{\partial}{\partial t}\right|_{t=0}\rho_t$.  In particular, the conclusion of \cref{thm:hirachi_ellipsoid} also holds for some of the ellipsoids considered by Reiter and Son~\cite{ReiterSon2019}; see \cref{rk:reiter_son} for further discussion.

\begin{proof}[Proof of \cref{thm:hirachi_ellipsoid}]
 Let $\rho_t$ be given by \cref{eqn:ellipsoid_deformation}. 
 Let $(S^t)_{\alpha\bar\beta\gamma\bar\sigma}$ denote the pullback of the Chern tensor of $(M_t,T_t^{1,0},\theta_t)$ by $F_t$.  Since $\rho_0$ is the defining function of the round CR $(2n+1)$-sphere, it holds that $(S^0)_{\alpha\bar\beta\gamma\bar\sigma}=0$.  This yields \cref{eqn:hirachi_ellipsoid_subcritical_derivative}.
 
 Recall that, on the round CR $(2n+1)$-sphere,
 \begin{equation}
  \label{eqn:c-to-w}
  \begin{split}
   w_{\alpha\beta} & = 0, \\
   w_{\alpha\bar\beta} & = -wh_{\alpha\bar\beta}, \\
   w_{\bar\beta} & = 0 , \\
   w^\gamma w_\gamma & = 1-\lv w\rv^2.
  \end{split}
 \end{equation}
 (One can deduce these formulae using the fact that $\frac{1}{\lv 1+w\rv^2}\theta$ on $S^{2n+1}\setminus\{w=-1\}$ equals the pullback of the standard contact form on the Heisenberg group under Cayley transform~\cite{JerisonLee1987} and the transformation laws~\cite{Lee1988} for the pseudohermitian curvature and torsion.) 
 Since $\dot\rho_t = \frac{1}{4}\lv w\rv^4$, we conclude from \cref{lem:linearize_chern} that
 \begin{equation}
  \label{eqn:perturb_S}
  \dot S_{\alpha\bar\beta\gamma\bar\sigma}  = \tf w_\alpha \ow_{\bar\beta} w_\gamma \ow_{\bar\sigma} .
 \end{equation}
 (We emphasize that $\alpha,\bar\beta,\gamma,\bar\sigma$ are abstract indices in this formula.)  \Cref{eqn:divS} then implies that
 \begin{equation}
  \label{eqn:perturb_V}
  \dot V_{\alpha\bar\beta\gamma} := \biggl(\left.\frac{\partial}{\partial t}\right|_{t=0} V^t\biggr)_{\alpha\bar\beta\gamma} = -\frac{n+3}{n+2}i\bar w\tf w_\alpha \ow_{\bar\beta} w_\gamma ,
 \end{equation}
 where $V^t:=F_t^\ast V^{\theta_t}$ and
 \[ \tf u_{\alpha\bar\beta\gamma} := u_{\alpha\bar\beta\gamma} - \frac{1}{n+1}\left(u_\mu{}^\mu{}_\alpha h_{\gamma\bar\beta} + u_\mu{}^\mu{}_\gamma h_{\alpha\bar\beta} \right) . \]

 Define
 \begin{align*}
  \dot\mC_\alpha{}^\beta & := i\dot S_\alpha{}^\beta{}_\mu{}^\nu \theta^\mu\theta_\nu, \\
  \dot\mV_\alpha{}^\beta & := \dot V_\alpha{}^\beta{}_\mu \theta^\mu ,
 \end{align*}
 where products are taken in the exterior algebra $\Lambda^\bullet S^{2n+1}$.  It follows from \cref{eqn:perturb_S} and \cref{eqn:perturb_V} that
 \begin{align}
  \label{eqn:Sform} \dot\mC_\alpha{}^\beta & = W_\alpha{}^\beta +  \frac{n+1}{n+2}c\left( \Psi_\alpha{}^\beta + M_\alpha{}^\beta + \delta_\alpha^\beta\bigl(i(\db w)(\dbbar\ow) + c\,d\theta\bigr)\right) , \\
  \label{eqn:Vform} \dot\mV_\alpha{}^\beta & = -\frac{n+3}{n+2}i\ow\left( (w_\alpha \ow^\beta + c\delta_\alpha^\beta)\db w + cw_\alpha\theta^\beta \right) ,
 \end{align}
 respectively, where
 \begin{align*}
  W_\alpha{}^\beta & := iw_\alpha \ow^\beta (\db w)(\dbbar\ow), \\
 \Psi_\alpha{}^\beta & := w_\alpha \ow^\beta d\theta + iw_\alpha\theta^\beta(\dbbar\ow) + i\ow^\beta(\db w)\theta_\alpha, \\
  M_\alpha{}^\beta & := ic\theta^\beta\theta_\alpha, \\
  c & := -\frac{1}{n+1}w_\gamma\ow^\gamma . 
 \end{align*}
 We break the computation into four steps.

 \begin{step}
  Compute powers of $\dot\mC_\alpha{}^\beta$. 
 \end{step}
 
 Observe that
 \begin{align*}
  \Psi_\alpha{}^\gamma\Psi_\gamma{}^\beta & = W_\alpha{}^\beta d\theta - (n+1)c\Psi_\alpha{}^\beta d\theta + (n+1)iM_\alpha{}^\beta(\db w)(\dbbar\ow) , \\
  \Psi_\alpha{}^\gamma M_\gamma{}^\beta & = M_\alpha{}^\gamma\Psi_\gamma{}^\beta = -iM_\alpha{}^\beta(\db w)(\dbbar\ow) , \\
  M_\alpha{}^\gamma M_\gamma{}^\beta & = -cM_\alpha{}^\beta d\theta .
 \end{align*}
 Combining this with the identities
 \begin{align*}
  W_\alpha{}^\gamma W_\gamma{}^\beta & = 0, \\
  W_\alpha{}^\gamma \Psi_\gamma{}^\beta & = \Psi_\alpha{}^\gamma W_\gamma{}^\beta = -(n+1)cA_\alpha{}^\beta d\theta, \\
  W_\alpha{}^\gamma M_\gamma{}^\beta & = M_\alpha{}^\gamma W_\gamma{}^\beta = 0 , \\
  W_\alpha{}^\beta(\db w) & = W_\alpha{}^\beta(\dbbar\ow)  = 0
 \end{align*}
 yields
 \begin{align*}
  (\Psi^k)_\alpha{}^\beta & = \left(-c\right)^{k-2}\Bigl[ (k-1)(n+1)^{k-2}W_\alpha{}^\beta d\theta \\
   & \qquad\qquad - (n+1)^{k-1}c\Psi_\alpha{}^\beta d\theta + (n+1)^{k-1}iM_\alpha{}^\beta (\db w)(\dbbar\ow) \Bigr] d\theta^{k-2}, \\
  (M^k)_\alpha{}^\beta & = (-c)^{k-1}M_\alpha{}^\beta d\theta^{k-1}
 \end{align*}
 for all $k\geq2$, where $(\Psi^k)_\alpha{}^\beta:=\Psi_\alpha{}^{\gamma_2}\Psi_{\gamma_2}{}^{\gamma_3}\dotsm\Psi_{\gamma_k}{}^\beta$.  It follows that
 \[ (M^j)_\alpha{}^\gamma (\Psi^k)_\gamma{}^\beta = 0 \]
 for all $j\geq1$ and $k\geq2$.  In particular,
 \begin{equation}
  \label{eqn:Psi+Mk}
  \begin{split}
   \bigl((\Psi+M)^k\bigr)_\alpha{}^\beta & = (\Psi^k)_\alpha{}^\beta + k\Psi_\alpha{}^\gamma (M^{k-1})_\gamma{}^\beta + (M^k)_\alpha{}^\beta \\
   & = (-c)^{k-2}\Bigl[ (k-1)(n+1)^{k-2}W_\alpha{}^\beta d\theta - (n+1)^{k-1}c\Psi_\alpha{}^\beta d\theta  \\
    & \qquad - cM_\alpha{}^\beta d\theta + \left((n+1)^{k-1}-k\right)iM_\alpha{}^\beta(\db w)(\dbbar\ow) \Bigr] d\theta^{k-2}
   \end{split}
 \end{equation}
 for all $k\geq1$.  Define
 \[ \mP_\alpha{}^\beta := \Psi_\alpha{}^\beta + M_\alpha{}^\beta + \delta_\alpha^\beta\left(i(\db w)(\dbbar\ow) + c\,d\theta\right) . \]
 Observe that
 \[ (\mP^k)_\alpha{}^\beta = \sum_{j=0}^k \binom{k}{j}c^{k-j-1}\bigl((\Psi+M)^j\bigr)_\alpha{}^\beta\bigl(c\,d\theta + (k-j)i(\db w)(\dbbar\ow)\bigr)d\theta^{k-j-1} ; \]
 as $\left(c\,d\theta+(k-j)i(\db w)(\dbbar\ow)\right)d\theta^{k-j-1}=c\,d\theta^{k-j} + (k-j)i(\db w)(\dbbar\ow)d\theta^{k-j-1}$, we interpret this factor as multiplication by the scalar function $c$ when $j=k$ in the summation.  Since $i\Psi_\alpha{}^\beta(\db w)(\dbbar\ow) = W_\alpha{}^\beta d\theta$, we conclude from \cref{eqn:Psi+Mk} that
 \begin{align*}
  (\mP^k)_\alpha{}^\beta & = c^{k-1}\delta_\alpha^\beta\bigl(c\,d\theta + ki(\db w)(\dbbar\ow)\bigr)\,d\theta^{k-1} \\
   & \quad + \sum_{j=1}^k (-1)^{j-1}\binom{k}{j}c^{k-1}\bigl[ (n+1)^{j-1}\Psi_\alpha{}^\beta + M_\alpha{}^\beta \bigr] d\theta^{k-1} \\
   & \quad + i\sum_{j=1}^k (-1)^j\binom{k}{j}c^{k-2}\left((n+1)^{j-1}-k\right)M_\alpha{}^\beta(\db w)(\dbbar\ow)d\theta^{k-2} \\
   & \quad + \sum_{j=1}^k(-1)^j\binom{k}{j}(n+1)^{j-2}c^{k-2}\left((n+2)j - 1 - k(n+1)\right)W_\alpha{}^\beta d\theta^{k-1}
 \end{align*}
 for all $k\geq 1$, where we adopt the convention $d\theta^{-1}:=0$. 
 Evaluating the summations yields
 \begin{equation}
  \label{eqn:mPk}
  \begin{split}
   (\mP^k)_\alpha{}^\beta & = c^{k-2}\Bigl[ c^2\delta_\alpha^\beta d\theta^2 + ikc\delta_\alpha^\beta(\db w)(\dbbar\ow)d\theta - \frac{(-n)^k-1}{n+1}c\Psi_\alpha{}^\beta d\theta \\
    & \qquad + cM_\alpha{}^\beta d\theta + \frac{(-n)^k-1+k(n+1)}{n+1}iM_\alpha{}^\beta(\db w)(\dbbar\ow) \\
    & \qquad + \left(\frac{k\left(1-2(-n)^{k-1}\right)}{n+1} + \frac{1-(-n)^k}{(n+1)^2}\right)W_\alpha{}^\beta d\theta \Bigr] d\theta^{k-2}
  \end{split}
 \end{equation}
 for all $k\geq 1$, where we distribute the multiplication by $d\theta^{k-2}$ and use our convention $d\theta^{-1}=0$ to make sense of the case $k=1$.
 
 Finally, since $W_\alpha{}^\beta$ and $\mP_\alpha{}^\beta$ commute, we have that
 \[ (\dot\mC^k)_\alpha{}^\beta = \left(\frac{n+1}{n+2}c\right)^{k-1}\left[ \frac{n+1}{n+2}c(\mP^k)_\alpha{}^\beta + kW_\alpha{}^\gamma(\mP^{k-1})_\gamma{}^\beta \right] . \]
 Combining this with \cref{eqn:mPk} yields
 \begin{equation}
  \label{eqn:mSk}
  \begin{split}
  (\dot\mC^k)_\alpha{}^\beta & = \left(\frac{n+1}{n+2}\right)^{k-1}c^{2k-2}\Bigl[ \frac{n+1}{n+2}c(c\delta_\alpha^\beta d\theta + M_\alpha{}^\beta) d\theta \\
   & \qquad + \frac{n+1}{n+2}kic\delta_\alpha^\beta(\db w)(\dbbar\ow)d\theta - \frac{(-n)^k-1}{n+2}c\Psi_\alpha{}^\beta d\theta \\
   & \qquad + \frac{(-n)^k-1+k(n+1)}{n+2}iM_\alpha{}^\beta(\db w)(\dbbar\ow) \\
   & \qquad + \frac{\left(k(n+1)+1\right)\left(1-(-n)^k\right)}{(n+1)(n+2)}W_\alpha{}^\beta d\theta \Bigr] d\theta^{k-2}
  \end{split}
 \end{equation}
 for all $k\geq1$.  Using the facts
 \begin{align*}
  W_\gamma{}^\gamma & = -(n+1)ic(\db w)(\dbbar\ow), \\
  M_\gamma{}^\gamma & = c\,d\theta , \\
  \Psi_\gamma{}^\gamma & = -(n+1)c\,d\theta + 2i(\db w)(\dbbar\ow),
 \end{align*}
 we conclude that
 \begin{equation}
  \label{eqn:tr_mSk}
  \tr\dot\mC^k = \left(\frac{n+1}{n+2}\right)^k\left(n+(-n)^{k}\right)c^{2k-1} \left( c\,d\theta + ki(\db w)(\dbbar\ow) \right) d\theta^{k-1}
 \end{equation}
 for all $k\geq1$.

 \begin{step}
  Compute derivatives of $c_\Phi(S)$ and $\nabla_\alpha c_\Phi(S)$.
 \end{step}

 Recall that
 \[ \frac{1}{n!}c_\Phi(S^t)d\theta^n = c_\Phi\left(i(S^t)_\alpha{}^\beta{}_{\mu\bar\nu}\theta^\mu\theta^{\bar\nu}\right) . \]
 It follows that that
 \[ \frac{1}{(n!)^2}\left.\frac{\partial^n}{\partial t^n}\right|_{t=0} c_\Phi(S^t)\,d\theta^n = c_\Phi(\dot\mC_\alpha{}^\beta) . \]
 Using our assumption $\Phi=\Phi(\varsigma)$, $\varsigma_1=0$, we have that
 \[ c_\Phi(\dot\mC_\alpha{}^\beta) = \prod_{k=2}^n(\tr\dot\mC^k)^{\varsigma_k} . \]
 Using~\eqref{eqn:tr_mSk} and the fact $ni(\db w)(\dbbar\ow)d\theta^{n-1}=-(n+1)c\,d\theta^n$, we deduce that
 \[ c_\Phi(\dot\mC_\alpha{}^\beta) = -n\left(\frac{n+1}{n+2}\right)^np(\varsigma)c^{2n}\,d\theta^n , \]
 where
 \begin{equation}
  \label{eqn:prod_varsigma}
  p(\varsigma) := \prod_{k=2}^n \left(n+(-n)^k\right)^{\varsigma_k} .
 \end{equation}
 Since $\nabla_\alpha c^{2n} = \frac{2n}{n+1}c^{2n-1}\ow w_\alpha$, we deduce that
 \begin{equation}
  \label{eqn:nabla_cPhiS}
  \frac{1}{(n!)^2}\left.\frac{\partial^n}{\partial t^n}\right|_{t=0}\nabla_\alpha c_\Phi(S^t) = -\frac{2n^2}{n+1}\left(\frac{n+1}{n+2}\right)^{n}p(\varsigma)c^{2n-1}\ow w_\alpha .
 \end{equation}
 
 \begin{step}
  Compute $i(\dot\mC^\Phi)_\alpha{}^\beta{}_\mu{}^\nu \dot V_\beta{}^\mu{}_\nu$.
 \end{step}

 Observe that
 \begin{align*}
  \dot\mV_\alpha{}^\alpha & = 0 , \\
  W_\alpha{}^\beta \dot\mV_\beta{}^\alpha & = 0, \\
  M_\alpha{}^\beta \dot\mV_\beta{}^\alpha & = 0, \\
  \Psi_\alpha{}^\beta \dot\mV_\beta{}^\alpha & = -\frac{n(n+1)(n+3)}{n+2}ic^2\ow(\db w)d\theta .
 \end{align*}
 Combining this with \cref{eqn:mSk} yields
 \begin{multline}
  \label{eqn:mSkmV}
  \tr \dot\mC^{k-1}\dot\mV := (\dot\mC^{k-1})_\alpha{}^\beta\dot\mV_\beta{}^\alpha \\ = -\left(\frac{n+1}{n+2}\right)^{k}\frac{(n+3)\bigl(n+(-n)^{k}\bigr)}{n+1}ic^{2k-1}\ow(\db w)d\theta^{k-1}
 \end{multline}
 for all $k\geq1$.

 We now compute the derivatives in $t$ of $i(\mS^\Phi)_\alpha{}^\beta{}_\mu{}^\nu V_\beta{}^\mu{}_\nu$ using the fact that
 \[ \frac{1}{(n-1)!}i(\mS^\Phi)_\alpha{}^\beta{}_\mu{}^\nu V_\beta{}^\mu{}_\nu \theta^\alpha d\theta^{n-1} = ic_{\Phi,n-1}(iS_\alpha{}^\beta{}_{\mu\bar\nu}\theta^\mu\theta^{\bar\nu},V_\alpha{}^\beta{}_\gamma\theta^\gamma) \]
 for all $t$, where 
 \[ c_{\Phi,n-1}(Y_\alpha{}^\beta,Z_\alpha{}^\beta) := \Phi_{\alpha_1\dotsm\alpha_n}^{\beta_1\dotsm\beta_n}Y_{\beta_1}{}^{\alpha_1}\dotsm Y_{\beta_{n-1}}{}^{\alpha_{n-1}}Z_{\beta_n}{}^{\alpha_n} \]
 for all invariant polynomials $\Phi$ of degree $n$, all $\End(T^{1,0})$-valued two-forms $Y_\alpha{}^\beta$, and all $\End(T^{1,0})$-valued one-forms $Z_\alpha{}^\beta$.  Note that if $\Phi=\Phi(\varsigma)$, then
 \[ c_{\Phi,n-1}(Y,Z) = \frac{1}{n}\sum_{k=2}^n \left[ k\varsigma_k(\tr Y^k)^{\varsigma_k-1}(\tr Y^{k-1}Z)\prod_{j\not=k}(\tr Y^j)^{\varsigma_j} \right] . \]
 Using \cref{eqn:tr_mSk,eqn:mSkmV} and our assumption $\Phi=\Phi(\varsigma)$, we compute that 
 \begin{align*}
  ic_{\Phi,n-1}(\dot\mC_\alpha{}^\beta,\dot\mV_\alpha{}^\beta) = \frac{n+3}{n+1}\left(\frac{n+1}{n+2}\right)^{n}p(\varsigma)c^{2n-1}\ow(\db w)d\theta^{n-1} .
 \end{align*}
 In particular,
 \begin{equation}
  \label{eqn:SV}
  \frac{i}{(n-1)!n!}\left.\frac{\partial^n}{\partial t^n}\right|_{t=0}(\mS_t^\Phi)_\alpha{}^\beta{}_\mu{}^\nu (V^t)_\beta{}^\mu{}_\nu = \frac{n+3}{n+1}\left(\frac{n+1}{n+2}\right)^{n}p(\varsigma)c^{2n-1}\ow w_\alpha .
 \end{equation}
 
 \begin{step}
  Compute derivatives of $X_\alpha^\Phi$ and $\Real\nabla^\alpha X_\alpha^\Phi$.
 \end{step}

 We now compute $X_\alpha^\Phi$ for $\Phi=\Phi(\varsigma)$.  \cref{eqn:nabla_cPhiS,eqn:SV} imply that
 \begin{equation}
  \label{eqn:time-derivative-of-X}
  \frac{1}{(n!)^2}\left.\frac{\partial^n}{\partial t^n}\right|_{t=0}F_t^\ast (X_\alpha^\Phi)^{\theta_t} = \frac{3}{n}\left(\frac{n+1}{n+2}\right)^{n}p(\varsigma)c^{2n-1} \ow w_\alpha .
 \end{equation}
 Using \cref{eqn:c-to-w}, we conclude that
 \[ \frac{1}{(n!)^2}\left.\frac{\partial^n}{\partial t^n}\right|_{t=0}F_t^\ast\left(\nabla^\alpha X_\alpha^\Phi\right)^{\theta_t} = -\frac{3}{n}\left(\frac{n+1}{n+2}\right)^n p(\varsigma)c^{2n-1}\left(3n(n+1)c+3n-1\right) . \]
 Since $p(\varsigma)\not=0$, we conclude that \cref{eqn:hirachi_ellipsoid_derivative} holds.
\end{proof}

\begin{remark}
 \label{rk:reiter_son}
 Reiter and Son~\cite[Equation (4.4)]{ReiterSon2019} computed the Chern tensor of the real ellipsoids
 \[ \Omega_s = \left\{ (w,z) \in \bC\times\bC^n \suchthat 1 - \lv z\rv^2 - \lv w\rv^2 - s\Real w^2 > 0 \right\} \]
 with respect to the unique pseudohermitian structure which is volume-normalized with respect to $dw\wedge dz^1\wedge\dotsm\wedge dz^n\rv_{\Omega_s}$.  Their computation shows that
 \begin{align*}
  \left. \frac{\partial}{\partial s}\right|_{s=0} F_s^\ast (S^{\theta_s})_{\alpha\bar\beta\gamma\bar\sigma} & = 0 , \\
  \left. \frac{\partial^2}{\partial s^2}\right|_{s=0} F_s^\ast (S^{\theta_s})_{\alpha\bar\beta\gamma\bar\sigma} & = 2\tf w_\alpha \ow_{\bar\beta} w_\gamma \ow_{\bar\sigma} ,
 \end{align*}
 where $F_s\colon\partial\Omega_0\to\partial\Omega_s$ is a one-parameter family of contact diffeomorphisms with $F_0=\Id$.  (Recall that we denote $w_\alpha=Z_\alpha w$, whereas Reiter and Son write their computation in terms of $Z_\alpha:=\partial_{z^\alpha} - \frac{(\rho_s)_\alpha}{(\rho_s)_w}\partial_w$, where $\rho_s$ is the given defining function for $\partial\Omega_s$.)  In particular, our proof of \cref{thm:hirachi_ellipsoid} shows that, for $s$ close to zero, the invariants $\mI_{\Phi(\varsigma)}^\prime$ on the real ellipsoids $\Omega_s$ give counterexamples to the Hirachi conjecture when $\varsigma_1=0$.
\end{remark}

\section{Counterexample via Calabi--Yau manifolds}
\label{sec:calabi-yau}

In this section, we prove the following result:

\begin{theorem}
\label{thm:counterexamples-via-calabi-yau}
	For $n \geq 2$,
	there exists a closed $(2n + 1)$-dimensional pseudo-Einstein manifold $(M, T^{1, 0}, \theta)$
	such that
	\begin{equation*}
		R_{\alpha \bar{\beta}}
		= 0,
		\qquad
		A_{\alpha \beta}
		= 0,
		\qquad
		\Real \nabla^{\alpha} X^{(n)}_{\alpha}
		\neq 0.
	\end{equation*}
\end{theorem}

We construct such a CR manifold as a certain circle bundle over a Calabi--Yau manifold.
Let $(Y, \omega)$ be an $n$-dimensional K\"{a}hler manifold.
There exists a smooth function $f_{\omega}$ on $Y$
such that the $n$-th Chern form $c_{n}(\omega)$ with respect to $\omega$
coincides with $f_{\omega} \cdot \omega^{n}$.

\begin{theorem}
\label{thm:Ricci-flat-with-non-constant-Chern-form}
    For each positive integer $n \geq 2$,
    there exists an $n$-dimensional closed, connected Ricci-flat K\"{a}hler manifold $(Y, \omega)$
    such that $[\omega / 2 \pi] \in H^{2}(Y; \mathbb{Z})$ and $f_{\omega}$ is non-constant.
\end{theorem}

\begin{proof}[Proof of \cref{thm:counterexamples-via-calabi-yau}
assuming \cref{thm:Ricci-flat-with-non-constant-Chern-form}]
	Since $[\omega / 2 \pi] \in H^{2}(Y; \mathbb{Z})$ and $Y$ is K\"{a}hler,
	there exists a holomorphic Hermitian line bundle $(L, h)$ over $Y$ such that $\omega = - i \Theta_{h}$.
	Consider the circle bundle $(M, T^{1, 0}, \theta)$ associated with $(Y, L, h)$.
	Since $\omega$ is Ricci-flat,
	$\theta$ is a contact form satisfying
	\begin{equation*}
		R_{\alpha \bar{\beta}} = 0, \qquad A_{\alpha \beta} = 0;
	\end{equation*}
	in particular,
	$V_{\alpha \bar{\beta} \gamma} = 0$.
	Moreover,
	$S_{\alpha\bar\beta\gamma\bar\sigma} = R_{\alpha\bar\beta\gamma\bar\sigma}$,
	and hence $c_{(n)}(S)$ is a nonzero constant multiple of $f_{\omega}$.
	In particular,
	$c_{(n)}(S)$ is non-constant.
	Therefore
	\begin{equation*}
		\Real \nabla^{\alpha} X^{(n)}_{\alpha}
		= - \frac{1}{n^{2}} \Real \nabla^{\alpha} \nabla_{\alpha} c_{(n)}(S)
		= - \frac{1}{2 n^{2}} \Delta_{b} c_{(n)}(S)
		\neq 0,
	\end{equation*}
	which completes the proof.
\end{proof}

It remains to show \cref{thm:Ricci-flat-with-non-constant-Chern-form}.

\begin{proof}[Proof of \cref{thm:Ricci-flat-with-non-constant-Chern-form}]
	As we will see in the following two subsections,
	such $(Y, \omega)$ exists in the cases of $n = 2$ and $3$.
	Since the conditions in \cref{thm:Ricci-flat-with-non-constant-Chern-form} are closed under the product,
	we can construct $(Y, \omega)$ for any $n \geq 2$.
\end{proof}

\subsection{Two-dimensional case}

Consider the two-dimensional complex torus $T = \mathbb{C}^{2} / (\mathbb{Z} + i \mathbb{Z})^{2}$.
Multiplication by $- 1$ on $\mathbb{C}^{2}$ induces an involution $\iota$ on $T$
that has $16$ fixed points $p_{1}, \dots , p_{16}$.
Let $\sigma \colon T^{\prime} \to T$ be obtained from $T$ by blowing up at $p_{1}, \dots, p_{16}$.
The involution $\iota$ lifts to an involution $\iota^{\prime}$ on $T^{\prime}$,
and the quotient $p \colon T^{\prime} \to Y = T^{\prime} / \langle \iota^{\prime} \rangle$ is a closed $K3$ surface;
this is called the \emph{Kummer surface associated to} $T$~\cite[Chapter V.16]{BarthHulekPetersVandeVen2004}.
The space $Y$ contains $16$ complex projective curves $E_{1}, \dots , E_{16}$ corresponding to $p_{1}, \dots , p_{16}$.
It is known that the Euler characteristic of any $K3$ surface is equal to $24$~\cite[Chapter VIII.3]{BarthHulekPetersVandeVen2004}.

Let $\omega^{(0)}$ be the K\"{a}hler form on $T$
induced by $2 \pi i \sum_{j = 1}^{2} d z^{j} \wedge d \overline{z}^{j}$ on $\mathbb{C}^{2}$;
the coefficient is chosen so that $[\omega^{(0)} / 2 \pi] \in H^{2}(T; \mathbb{Z})$.
For $0 < s \ll 1$,
the cohomology class $p_{!} \sigma^{*} [\omega^{(0)}] - s \sum_{k = 1}^{16} c_{1}(\mathcal{O}(E_{k}))$
contains a unique Ricci-flat K\"{a}hler metric $\omega_{s}$ on $Y$
such that $\omega_{s}$ converges smoothly to a flat K\"{a}hler metric as $s \to +0$
on any compact subset of $Y \setminus \bigcup_{k = 1}^{16} E_{k}$~\cite[Chapter 2]{Kobayashi1990}.
Note that
\begin{equation*}
	0 < \int_{Y} \omega_{s}^{2}
	= \int_{Y} (p_{!} \sigma^{*} [\omega^{(0)}])^{2} - 32 s^{2}
	< \int_{Y} (p_{!} \sigma^{*} [\omega^{(0)}])^{2}.
\end{equation*}
Suppose that $f_{s} := f_{\omega_{s}}$ is constant for any $0 < s \ll 1$.
From the Gauss--Bonnet--Chern formula, it follows that
\begin{equation*}
	24
	= \int_{Y} c_{2} (\omega_{s})
	= f_{s} \int_{Y} \omega_{s}^{2}
	< f_{s} \int_{Y} (p_{!} \sigma^{*} [\omega^{(0)}])^{2};
\end{equation*}
that is,
\begin{equation*}
	f_{s}
	> 24 \left[ \int_{Y} (p_{!} \sigma^{*} [\omega^{(0)}])^{2} \right]^{- 1}.
\end{equation*}
However,
since $\omega_{s}$ converges smoothly to a flat K\"{a}hler metric as $s \to +0$
on any compact subset of $Y \setminus \bigcup_{k = 1}^{16} E_{k}$,
we have $f_{s} \ll 1$ for sufficiently small $s$;
this is a contradiction.
Hence $f_{s}$ is non-constant for sufficiently small $s$.
If we take a sufficiently large positive integer $N$,
the Ricci-flat K\"{a}hler metric $\omega = N \cdot \omega_{2 \pi / N}$ satisfies
$[\omega / 2 \pi] \in H^{2}(Y; \mathbb{Z})$ and $f_{\omega}$ is non-constant.

\subsection{Three-dimensional case}

Let $\zeta$ be a primitive cubic root of one,
and denote by $E_{\zeta}$ the elliptic curve $\mathbb{C} / (\mathbb{Z} + \mathbb{Z} \zeta)$.
Multiplication by $\zeta$ on $\mathbb{C}^{3}$
induces a biholomorphism $\Phi_{\zeta}$ on $E_{\zeta}^{3}$.
This map satisfies $\Phi_{\zeta}^{3} = \Id$ and has $27$ fixed points $p_{1}, \dots , p_{27}$.
Let $\sigma \colon \widetilde{Y} \to E_{\zeta}^{3}$ be obtained from $E_{\zeta}^{3}$
by blowing up at $p_{1}, \dots, p_{27}$.
The biholomorphism $\Phi_{\zeta}$ lifts to a biholomorphism $\widetilde{\Phi}_{\zeta}$ on $\widetilde{Y}$
satisfying $\widetilde{\Phi}_{\zeta}^{3} = \Id$,
and the quotient $p \colon \widetilde{Y} \to Y = \widetilde{Y} / \langle \widetilde{\Phi}_{\zeta} \rangle$
is a closed smooth Calabi--Yau threefold,
called a \emph{Kummer threefold}.
The space $Y$ contains $27$ complex projective planes $E_{1}, \dots , E_{27}$ corresponding to $p_{1}, \dots , p_{27}$.
The Euler characteristic of $Y$ is $72$~\cite[Theorem 5(i)]{RoanYau1987}.

Let $\omega^{(0)}$ be the K\"{a}hler form on $E_{\zeta}^{3}$
induced by $(2 \pi \sqrt{3}) i \sum_{j = 1}^{3} d z^{j} \wedge d \overline{z}^{j}$ on $\mathbb{C}^{3}$;
the coefficient is chosen so that $[\omega^{(0)} / 2 \pi] \in H^{2}(E_{\zeta}^{3}; \mathbb{Z})$.
For $0 < s \ll 1$,
the cohomology class $p_{!} \sigma^{*} [\omega^{(0)}] - s \sum_{k = 1}^{27} c_{1}(\mathcal{O}(E_{k}))$
contains a unique Ricci-flat K\"{a}hler metric $\omega_{s}$ on $Y$
such that $\omega_{s}$ converges in $C^{4, \alpha}$ to a flat K\"{a}hler metric as $s \to +0$
on any compact subset of $Y \setminus \bigcup_{k = 1}^{27} E_{k}$~\cite[Section 3.1]{Lu1999}.
Note that
\begin{equation*}
	0
	< \int_{Y} \omega_{s}^{3}
	= \int_{Y} (p_{!} \sigma^{*} [\omega^{(0)}])^{3} - 243 s^{3}
	< \int_{Y} (p_{!} \sigma^{*} [\omega^{(0)}])^{3}.
\end{equation*}

Suppose that $f_{s} := f_{\omega_{s}}$ is constant for any $0 < s \ll 1$.
From the Gauss--Bonnet--Chern formula,
it follows that
\begin{equation*}
	72
	= \int_{Y} c_{3} (\omega_{s})
	= f_{s} \int_{Y} \omega_{s}^{3}
	< f_{s} \int_{Y} (p_{!} \sigma^{*} [\omega^{(0)}])^{3};
\end{equation*}
that is,
\begin{equation*}
	f_{s}
	> 72 \left[ \int_{Y} (p_{!} \sigma^{*} [\omega^{(0)}])^{3} \right]^{- 1}.
\end{equation*}
However,
since $\omega_{s}$ converges in $C^{4, \alpha}$ to a flat K\"{a}hler metric as $s \to +0$
on any compact subset of $Y \setminus \bigcup_{k = 1}^{27} E_{k}$,
we have $f_{s} \ll 1$ for sufficiently small $s$;
this is a contradiction.
Hence $f_{s}$ is non-constant for sufficiently small $s$.
If we take a sufficiently large positive integer $N$,
the Ricci-flat K\"{a}hler metric $\omega = N \cdot \omega_{2 \pi / N}$ satisfies
$[\omega / 2 \pi] \in H^{2}(Y; \mathbb{Z})$ and $f_{\omega}$ is non-constant.

\section{The $\mathcal{I}^{\prime}$-curvature of the boundary of a Reinhardt domain}
\label{sec:reinhardt}

For $r > 0$,
let $M_{r}$ be the boundary of the bounded Reinhardt domain
\begin{equation*}
	\Omega_{r}
	:= \{ w = (w^{0}, \dots , w^{n}) \in \mathbb{C}^{n + 1} \mid \rho_{r}(w) > 0 \},
\end{equation*}
where
\begin{equation*}
	\rho_{r}(w) := \frac{1}{2} - \frac{1}{2 r^{2}} \sum_{j = 0}^{n} (\log |w^{j}|)^{2}.
\end{equation*}
We would like to compute the total $\mathcal{I}^{\prime}$-curvatures for $M_{r}$.
To this end,
consider the holomorphic map
\begin{equation*}
	\psi_{r} \colon \mathbb{C}^{n + 1} \to \mathbb{C}^{n + 1};
	(z^{0}, \dots , z^{n}) \mapsto (\exp(2 r z^{0}), \dots , \exp(2 r z^{n})).
\end{equation*}
The pull-back $\psi_{r}^{*} \rho_{r} (z)$ coincides with
\begin{equation*}
	\rho(z)
	:= \frac{1}{2} - 2 \sum_{j = 0}^{n} (\Real z^{j})^{2},
\end{equation*}
and the pre-image of $\Omega_{r}$ by $\psi_{r}$ is the tube domain
\begin{equation*}
	\Omega
	= \left\{ z = (x^{0} + i y^{0}, \dots , x^{n} + i y^{n}) \in \mathbb{C}^{n + 1} \middle| \,
		|x|^{2} = \sum_{j = 0}^{n} (x^{j})^{2} < \frac{1}{4} \right\}
\end{equation*}
The holomorphic map $\psi_{r}$ induces also a pseudohermitian map
\begin{equation*}
	(M := \partial \Omega, T^{1, 0}, \theta := \Imaginary \overline{\partial} \rho |_{M})
	\to (M_{r}, T^{1, 0}_{r}, \theta_{r} := \Imaginary \overline{\partial} \rho_{r} |_{M_{r}}).
\end{equation*}
where $T^{1,0}:=T^{1,0}\bC^{n+1} \cap (TM\otimes\bC)$ and $T_r^{1,0}:=T^{1,0}\bC^{n+1}\cap(TM_r\otimes\bC)$.  The group $G = O(n + 1) \ltimes (i \mathbb{R})^{n + 1}$ acts on $\mathbb{C}^{n + 1}$
as a subgroup of the complex affine transformation group $GL(n + 1, \mathbb{C}) \ltimes \mathbb{C}^{n + 1}$,
and its action preserves $\rho$.
In particular,
the pseudohermitian manifold
$(M, T^{1, 0}, \theta)$
is homogeneous with respect to the above $G$-action.
Hence it suffices to consider a given point $p := (1 / 2, 0, \dots , 0) \in M$
for computing pseudo-Hermitian invariants.
We set $x' := (x^{1}, \dots , x^{n})$. 
Let
\begin{equation*}
	\xi
	:= - \frac{1}{2 |x|^{2}} \sum_{j = 0}^{n} x^{j} \frac{\partial}{\partial z^{j}}
	\in \Gamma(T^{1, 0} \mathbb{C}^{n + 1} |_{M}).
\end{equation*}
This vector field satisfies
\begin{equation*}
	\xi \rho
	= 1,
	\qquad
	\xi \contr \partial \overline{\partial} \rho
	= - \frac{1}{4 |x|^{2}} \overline{\partial} \rho.
\end{equation*}
For $\alpha \in \{1, \dots , n \}$,
the $(1, 0)$-forms
\begin{equation*}
	\theta^{\alpha}
	:= d z^{\alpha} + \frac{1}{2 |x|^{2}} x^{\alpha} \partial \rho
\end{equation*}
annihilate $\xi$
and their restriction to $M$ gives an admissible coframe.
A calculation shows that
the Levi form $h_{\alpha \bar{\beta}}$ is given by
\begin{equation*}
	h_{\alpha \bar{\beta}}
	= \delta_{\alpha \beta} + \frac{x^{\alpha} x^{\beta}}{(x^{0})^{2}}
	= \delta_{\alpha \beta} + 4 x^{\alpha} x^{\beta} + O(|x'|^{4}).
\end{equation*}
A similar computation to that in the proof of \cite[Proposition 5.2]{Marugame2016}
gives that
\begin{gather*}
	\omega_{\alpha}{}^{\beta}
	= - i \delta_{\alpha}{}^{\beta} \theta + 2 x^{\beta} \theta^{\alpha} + 2 x^{\beta} \theta^{\bar{\alpha}} + O(|x'|^{2}), \\
	A_{\alpha \beta}
	= - i \delta_{\alpha \beta} + O(|x'|^{2}).
\end{gather*}
At $p$,
the pseudohermitian torsion $A_{\alpha \beta}$ satisfies
\begin{equation*}
	\nabla_{\gamma} A_{\alpha \beta}
	= 0,
	\qquad
	\nabla_{\bar{\gamma}} A_{\alpha \beta}
	= 0,
	\qquad
	\nabla_{0} A_{\alpha \beta}
	= 2 i A_{\alpha \beta},
	\qquad
	A_{\alpha \beta} A^{\beta}{}_{\bar{\gamma}}
	= h_{\alpha \bar{\gamma}}.
\end{equation*}
Since both sides of these equalities are tensorial
and $(M, T^{1, 0}, \theta)$ is homogeneous,
these in fact hold on the whole of $M$.
Similarly,
the curvature form $\Pi_{\alpha}{}^{\beta}$ at $p$ is given by
\begin{equation*}
	\Pi_{\alpha}{}^{\beta}
	= (\delta_{\alpha}^{\beta} h_{\rho \bar{\sigma}} + \delta_{\rho}^{\beta} h_{\alpha \bar{\sigma}}
		- A_{\alpha \rho} A^{\beta}{}_{\bar{\sigma}}) \theta^{\rho} \wedge \theta^{\bar{\sigma}}
		- i \tau_\alpha \wedge \theta^{\beta}
		+ i \theta_\alpha \wedge \tau^\beta.
\end{equation*}
The right hand side is tensorial,
and so this equality holds on the whole of $M$.
Local pseudohermitian invariants can be calculated explicitly:
\begin{gather*}
	P_{\alpha \bar{\beta}}
	= \frac{n}{2 (n + 1)} h_{\alpha \bar{\beta}}, \\
	S_{\alpha \bar{\beta} \gamma \bar{\sigma}}
	= \frac{1}{n + 1} (h_{\alpha \bar{\beta}} h_{\gamma \bar{\sigma}}
		+ h_{\alpha \bar{\sigma}} h_{\gamma \bar{\beta}})
		- A_{\alpha \gamma} A_{\bar{\beta} \bar{\sigma}}, \\
	V_{\alpha \bar{\beta} \gamma}
	= 0, \\
	U_{\alpha \bar{\beta}}
	= 0.
\end{gather*}
In particular,
$\theta$ (or $\theta_{r}$) is a pseudo-Einstein contact form
with constant scalar curvature but nonvanishing pseudohermitian torsion.
Moreover,
the Chern tensor is parallel;
\begin{equation*}
	\nabla_{\rho} S_{\alpha \bar{\beta} \gamma \bar{\sigma}}
	= 0,
	\qquad
	\nabla_{\bar{\rho}} S_{\alpha \bar{\beta} \gamma \bar{\sigma}}
	= 0,
	\qquad
	\nabla_{0} S_{\alpha \bar{\beta} \gamma \bar{\sigma}}
	= 0.
\end{equation*}

\begin{theorem}
\label{thm:I-prime-curvature-for-Reinhardt-domains}
	The total $\mathcal{I}^{\prime}_{\Phi(\varsigma)}$-curvature $\overline{\mathcal{I}}^{\prime}_{\Phi(\varsigma)}$ for $M_{r}$
	is given by
	\begin{equation*}
	\label{eqn:total-I-curv-of-reinhardt}
		\overline{\mathcal{I}}^{\prime}_{\Phi(\varsigma)}
		= - (n !)^{2} \Vol(S^{n}(1))  \left( \frac{2 \pi}{(n + 1)r} \right)^{n + 1}
			\prod_{k = 1}^{n} [(n + 2)(1 - (n + 2)^{k - 1})]^{\varsigma_{k}},
	\end{equation*}
	where $\Vol(S^{n}(1))$ is the volume of the unit sphere in $\mathbb{R}^{n + 1}$.
\end{theorem}

\begin{proof}
	Set
	\begin{equation*}
		\Sigma_{\alpha}{}^{\beta}
		:= \frac{1}{n + 1} (- i \delta_{\alpha}^{\beta} d \theta + \theta^{\beta} \wedge \theta_{\alpha}),
		\qquad
		L_{\alpha}{}^{\beta}
		:= - \tau_{\alpha} \wedge \tau^{\beta},
	\end{equation*}
	which satisfy $\Xi_{\alpha}{}^{\beta} = \Sigma_{\alpha}{}^{\beta} + L_{\alpha}{}^{\beta}$, where $\Xi_\alpha{}^\beta$ is defined by \cref{eqn:Xi}.
	These $\Sigma_{\alpha}{}^{\beta}$ and $L_{\alpha}{}^{\beta}$ satisfy
	\begin{gather*}
		\tr \Sigma
		= - i d \theta,
		\qquad
		\tr L
		= i d \theta, \\
		\Sigma_{\alpha}{}^{\gamma} \wedge \Sigma_{\gamma}{}^{\beta}
		= - \frac{i}{n + 1} d \theta \wedge \Sigma_{\alpha}{}^{\beta}, \\
		L_{\alpha}{}^{\gamma} \wedge \Sigma_{\gamma}{}^{\beta}
		= \Sigma_{\alpha}{}^{\gamma} \wedge L_{\gamma}{}^{\beta}
		= - \frac{i}{n + 1} d \theta \wedge L_{\alpha}{}^{\beta}, \\
		L_{\alpha}{}^{\gamma} \wedge L_{\gamma}{}^{\beta}
		= - i d \theta \wedge L_{\alpha}{}^{\beta}.
	\end{gather*}
	Hence
	\begin{align*}
		(\Xi^{k})_{\alpha}{}^{\beta}
		&= (\Sigma^{k})_{\alpha}{}^{\beta}
			+ \sum_{j = 1}^{k} \binom{k}{j} (\Sigma^{k - j})_{\alpha}{}^{\gamma} \wedge (L^{j})_{\gamma}{}^{\beta} \\
		&= \frac{1}{(n + 1)^{k - 1}} (- i d \theta)^{k - 1} \wedge \Sigma_{\alpha}{}^{\beta}
			+ \sum_{j = 1}^{k} \binom{k}{j} \left( \frac{1}{n + 1} \right)^{k - j} (- i d \theta)^{k - 1} \wedge L_{\alpha}{}^{\beta}\\
		&= (- i d \theta)^{k - 1} \wedge \left[\frac{1}{(n + 1)^{k - 1}} \Sigma_{\alpha}{}^{\beta}
			+ \frac{(n + 2)^{k} - 1}{(n + 1)^{k}} L_{\alpha}{}^{\beta} \right],
	\end{align*}
	and so
	\begin{equation*}
		\tr \Xi^{k}
		= \frac{(n + 2)[1 - (n + 2)^{k - 1}]}{(n + 1)^{k}} (- i d \theta)^{k}.
	\end{equation*}
	Since
	\begin{equation*}
		c_{\Phi(\varsigma)}(i \Xi_{\alpha}{}^{\beta})
		= \frac{1}{n !} c_{\Phi(\varsigma)}(S) d \theta^{n},
	\end{equation*}
	we have
	\begin{align*}
		c_{\Phi(\varsigma)}(S)
		&= n ! \prod_{k = 1}^{n} \left[\frac{(n + 2)(1 - (n + 2)^{k - 1})}{(n + 1)^{k}} \right]^{\varsigma_{k}} \\
		&= \frac{n !}{(n + 1)^{n}} \prod_{k = 1}^{n} [(n + 2)(1 - (n + 2)^{k - 1})]^{\varsigma_{k}}.
	\end{align*}
	Therefore
	the $\mathcal{I}^{\prime}_{\Phi(\varsigma)}$-curvature of $M$ is given by
	\begin{align*}
		\mathcal{I}^{\prime}_{\Phi(\varsigma)}
		&= - \frac{n !}{(n + 1)^{n + 1}} \prod_{k = 1}^{n} [(n + 2)(1 - (n + 2)^{k - 1})]^{\varsigma_{k}}.
	\end{align*}
	In particular,
	$\mathcal{I}^{\prime}_{\Phi(\varsigma)}$ is constant,
	and equal to zero if and only if $\varsigma_{1} \neq 0$.

	We need also to compute the volume $\int_{M_{r}} \theta_{r} \wedge d \theta_{r}^{n}$.
	The pseudohermitian map $\psi_{r} \colon M \to M_{r}$ is a $\mathbb{Z}^{n}$-covering,
	and a fundamental domain $\Lambda_{r}$ is given by
	\begin{equation*}
		\Lambda_{r}
		:= \{ z = x + i y \in \mathbb{C}^{n + 1} \mid |x|^{2} = 1 / 4, y \in [0, \pi / r)^{n + 1} \}.
	\end{equation*}
	It suffices to compute the volume of $\Lambda_{r}$.
	From the definition of $\theta$,
	we have
	\begin{equation*}
		\theta
		= 2 \sum_{j = 0}^{n} x^{j} d y^{j},
		\qquad
		d \theta
		= 2 \sum_{j = 0}^{n} d x^{j} \wedge d y^{j}.
	\end{equation*}
	Hence
	\begin{align*}
		\int_{M_{r}} \theta_{r} \wedge d \theta_{r}^{n}
		&= \int_{\Lambda_{r}} \theta \wedge d \theta^{n} \\
		&= \sum_{j = 0}^{n} 2^{n + 1} n ! \int_{\Lambda_{r}} (d x^{0} \wedge d y^{0}) \wedge \dots \wedge (x^{j} d y^{j})
			\wedge \dots \wedge (d x^{n} \wedge d y^{n}) \\
		&= 2^{n + 1} n ! \left( \frac{\pi}{r} \right)^{n + 1}
			\int_{S^{n}(1/2)} \sum_{j = 0}^{n} (- 1)^{j} d x^{0} \wedge \dots \wedge x^{j} \wedge \dots \wedge d x^{n}.
	\end{align*}
	The $n$-form
	\begin{equation*}
		\sum_{j = 0}^{n} (- 1)^{j} d x^{0} \wedge \dots \wedge x^{j} \wedge \dots \wedge d x^{n}
	\end{equation*}
	on $S^{n}(1/2)$ is half of its volume form,
	and so
	\begin{equation*}
		\int_{S^{n}(1/2)} \sum_{j = 0}^{n} (- 1)^{j} d x^{0} \wedge \dots \wedge x^{j} \wedge \dots \wedge d x^{n}
		= 2^{- n - 1} \Vol(S^{n}(1)).
	\end{equation*}
	Therefore we have \cref{eqn:total-I-curv-of-reinhardt}.
\end{proof}

\section{Concluding remarks}
\label{sec:conclusion}

In light of Alexakis' characterization of global conformal invariants~\cite{Alexakis2012}, it is natural to expect that a weaker version of \cref{conj:strong_hirachi} is true.  One way to weaken \cref{conj:strong_hirachi} is to allow, in addition to local CR invariants, pseudohermitian scalar invariants $I$ for which $\mP\subset\ker D_\theta I$ for all pseudo-Einstein contact forms $\theta$.  We propose allowing an even weaker type of invariant.

\begin{definition}
 \label{defn:local_secondary_invariant}
 Fix $n\in\bN$.  A homogeneous pseudohermitian scalar invariant $I^\theta$ is a \emph{local secondary invariant} if
 \begin{equation}
  \label{eqn:local_secondary_invariant}
  \int_M uI^{\htheta}\,\htheta \wedge d\htheta^n = \int_M uI^\theta\,\theta \wedge d\theta^n
 \end{equation}
 for any pseudo-Einstein contact forms $\theta$ and $\htheta$ on a closed CR manifold $(M^{2n+1},T^{1,0})$ and any $u\in\mP$.
\end{definition}

Note that if $I$ is homogeneous of degree $-n-1$ in $\theta$ and if $\mP\subset\ker D_\theta I$ for all pseudo-Einstein contact forms $\theta$, then it is a local secondary invariant.  We propose the following weaker version of \cref{conj:strong_hirachi}.

\begin{conjecture}
 \label{conj:weak_hirachi}
 Let $I$ be a natural pseudohermitian scalar invariant whose total integral is a secondary CR invariant.  Then there is a constant $c\in\bR$ such that
 \[ I = cQ^\prime + (\textup{local secondary invariant}) + (\textup{divergence}) . \]
\end{conjecture}

There are two motivations behind \cref{defn:local_secondary_invariant}, and hence \cref{conj:weak_hirachi}.

Our first motivation is in analogy with the $Q^\prime$-curvature.  Let $\mP^\perp$ denote the space of smooth volume forms which annihilate $\mP$; i.e.\ given a closed pseudohermitian manifold $(M^{2n+1},T^{1,0},\theta)$, we set
\[ \mP^\perp := \left\{ \psi\,\theta\wedge d\theta^n \suchthat \text{$\int_{M} u\psi\,\theta\wedge d\theta^n = 0$ for all $u\in\mP$} \right\} . \]
Note that $\psi\,\theta\wedge d\theta^n\in\mP^\perp$ if and only if $\psi$ is $L^2$-orthogonal to $\mP$ with respect to $\theta$, so that this definition coincides with the definition of $\mP^\perp$ given in the introduction.  Since $\mP^\perp$ is CR invariant, \cref{defn:local_secondary_invariant} is equivalent to the requirement that $I^\theta\,\theta\wedge d\theta^n$ is independent of the choice of pseudo-Einstein contact form modulo $\mP^\perp$.  This is analogous to how one realizes the $Q^\prime$-curvature as having a linear transformation law when working modulo $\mP^\perp$; see \cref{eqn:Qprime-transformation}.

Our second motivation is speculation based on the compatibility of \cref{defn:local_secondary_invariant} with the heuristic construction of ``primed'' invariants by analytic continuation in the dimension (cf.\ \cite{CaseGover2013,CaseYang2012,Marugame2019}).  Suppose that $I$ is a family of local CR invariants of weight $-n-1$ defined on all CR manifolds of CR dimension $d\geq n$, and moreover suppose that $I^\theta=0$ for any pseudo-Einstein contact form in CR dimension $n$.  Suppose further that the formal limit
\begin{equation}
 \label{eqn:formal_limit}
 I^\prime = \lim_{d\to n} \frac{1}{d-n}I^\theta
\end{equation}
makes sense when restricted to pseudo-Einstein manifolds.  The fact that $I$ is CR invariant implies that
\[ \int_{M^{2d+1}} u\hI\,\htheta\wedge d\htheta^d = \int_{M^{2d+1}} uI\,\theta\wedge d\theta^d \]
for all closed CR manifolds $(M^{2d+1},T^{1,0})$, all contact forms on $(M,T^{1,0})$, and all (real) densities $u$ of weight $n-d$; i.e.\ all equivalence classes $u=[u,\theta]$ subject to the relation $[u,\theta]=[e^{(n-d)\Upsilon}u,e^\Upsilon\theta]$.  Dividing both sides by $d-n$, restricting to pseudo-Einstein contact forms, taking the limit $d\to n$, and restricting to CR pluriharmonic functions then implies that $I^\prime$ is a local secondary invariant.  The restriction to CR pluriharmonic functions is for symmetry reasons, as two pseudo-Einstein contact forms $\theta$ and $\htheta=e^\Upsilon\theta$ are necessarily such that $\Upsilon\in\mP$.

Unfortunately, none of our nontrivial $\mI_\Phi^\prime$-curvatures seem to be local secondary invariants in the sense of \cref{defn:local_secondary_invariant}.  This observaton arises from two heuristics.

First, the Case--Gover construction~\cite{CaseGover2013} of $\mI^\prime$ in CR dimension two arises from analytic continuation in the dimension after working modulo divergences.  Since working modulo divergences breaks CR invariance, we expect $\mI^\prime$ to only be a local secondary invariant modulo a divergence.  A similar interpretation to the higher-dimensional $\mI_\Phi^\prime$-curvatures was given by Marugame~\cite{Marugame2019}.

Second, the $\mI_\Phi^\prime$-curvatures can be realized via analytic continuation without working modulo divergences, but by starting with variational pseudohermitian scalar invariants:

Let $\Phi$ be an invariant polynomial of degree $n$ and let $(M^{2d+1},T^{1,0},\theta)$ be a pseudohermitian manifold of CR dimension $d$.  Define
\begin{align*}
 c_\Phi(S) & := \delta_{\alpha_1\dotsm\alpha_n}^{\beta_1\dotsm\beta_n}\Phi_{\mu_1\dotsm\mu_n}^{\nu_1\dotsm\nu_n} S_{\beta_1}{}^{\alpha_1}{}_{\nu_1}{}^{\mu_1} \dotsm S_{\beta_n}{}^{\alpha_n}{}_{\nu_n}{}^{\mu_n}, \\
 X_\alpha^\Phi & := i(\mS^\Phi)_\alpha{}^\beta{}_\mu{}^\nu V_\beta{}^\mu{}_\nu - \frac{1}{dn}\nabla_\alpha c_\Phi(S) , \\
 \mI_\Phi & := -\frac{2}{n}\mU_\alpha{}^\beta P_\beta{}^\alpha + (d-n) \biggl[ \frac{1}{dn(2n-d)}\left( \Delta_b c_\Phi(S) - 2nPc_\Phi(S)\right) \\
  & \qquad + (\mT^\Phi)_\alpha{}^\beta{}_{\mu_1}{}^{\nu_1}{}_{\mu_2}{}^{\nu_2}\left((n-1)V_\beta{}^{\mu_1}{}_{\nu_1} V^\alpha{}_{\nu_2}{}^{\mu_2} - S_\beta{}^\alpha{}_{\nu_1}{}^{\mu_1} U_{\nu_2}{}^{\mu_2}\right) \biggr] ,
\end{align*}
where
\begin{align*}
 (\mS^\Phi)_\alpha{}^\beta{}_\mu{}^\nu & := \delta_{\alpha\alpha_2\dotsm\alpha_n}^{\beta\beta_2\dotsm\beta_n}\Phi_{\mu\mu_2\dotsm\mu_n}^{\nu\nu_2\dotsm\nu_n} S_{\beta_2}{}^{\alpha_2}{}_{\nu_2}{}^{\mu_2}\dotsm S_{\beta_n}{}^{\alpha_n}{}_{\nu_n}{}^{\mu_n}, \\
 (\mT^\Phi)_\alpha{}^\beta{}_{\mu_1}{}^{\nu_1}{}_{\mu_2}{}^{\nu_2} & := \delta_{\alpha\alpha_3\dotsm\alpha_n}^{\beta\beta_3\dotsm\beta_n}\Phi_{\mu_1\dotsm\mu_n}^{\nu_1\dotsm\nu_n} S_{\beta_3}{}^{\alpha_3}{}_{\nu_3}{}^{\mu_3} \dotsm S_{\beta_n}{}^{\alpha_n}{}_{\nu_n}{}^{\mu_n}, \\
 \mU_\alpha{}^\beta & := \delta_{\alpha\alpha_1\dotsm\alpha_n}^{\beta\beta_1\dotsm\beta_n}\Phi_{\mu_1\dotsm\mu_n}^{\nu_1\dotsm\nu_n} S_{\beta_1}{}^{\alpha_1}{}_{\nu_1}{}^{\mu_1} \dotsm S_{\beta_n}{}^{\alpha_n}{}_{\nu_n}{}^{\mu_n} - \frac{d-n}{d}\delta_\alpha^\beta c_\Phi(S) .
\end{align*}
Note that when $d=n$, each of $c_\Phi(S)$, $X_\alpha^\Phi$, $(\mS^\Phi)_\alpha{}^\beta{}_\mu{}^\nu$, and $(\mT^\Phi)_\alpha{}^\beta{}_{\mu_1}{}^{\nu_1}{}_{\mu_2}{}^{\nu_2}$ recovers our original definitions given in the introduction.  Moreover, note that $\mU_\alpha{}^\beta$ is trace-free for all $d$ and that $\mU_\alpha{}^\beta=0$ when $d=n$.  These observations imply that $\mU_\alpha{}^\beta P_\beta{}^\alpha=0$ on all pseudo-Einstein manifolds.  Indeed, by restricting $\mI_\Phi$ to pseudo-Einstein manifolds and formally taking a dimensional limit, we have that
\begin{equation}
 \label{eqn:I-prime-as-limit}
 \lim_{d\to n} \frac{1}{d-n}\mI_\Phi = \mI_\Phi^\prime ;
\end{equation}
that is, the $\mI_\Phi^\prime$-curvature can be interpreted as the secondary invariant associated to $\mI_\Phi$ via analytic continuation in the dimension, analogous to the heuristic interpretation of the $Q^\prime$-curvature~\cite{CaseYang2012,Hirachi2013}.

One nice property of $\mI_\Phi$ is that it is a variational pseudo-Einstein invariant.  More precisely, using the identity
\begin{equation}
 \label{eqn:mC-divergence}
 \nabla_\beta \mU_\alpha{}^\beta = n(d-n)X_\alpha^\Phi,
\end{equation}
it is straightforward to compute that
\begin{equation}
 \label{eqn:mC-linearization}
 e^{(n+1)\Upsilon} \hmI_\Phi = \mI_\Phi - \frac{2}{n}\Real\nabla^\gamma\left(\mU_\gamma{}^\beta\Upsilon_\beta\right)
\end{equation}
for all pseudohermitian manifolds $(M^{2d+1},T^{1,0},\theta)$ and all $\htheta:=e^\Upsilon\theta$, $\Upsilon\in C^\infty(M)$.  It follows that
\begin{equation}
 \label{eqn:mI-variational}
 \left.\frac{d}{dt}\right|_{t=0} \int_M (\mI_\Phi)^{\theta_t}\,\theta_t\wedge d\theta_t^d = (d-n)\int_M \mI_\Phi \Upsilon \, \theta \wedge d\theta^d
\end{equation}
for all one-parameter families $\theta_t=e^{t\Upsilon}\theta$ of contact forms on $(M^{2d+1},T^{1,0})$.

Together with the realization of $\mI_\Phi^\prime$ as the limit of \cref{eqn:I-prime-as-limit}, the previous paragraph suggests that the $\mI_\Phi^\prime$-curvature should be variational in the space of pseudo-Einstein contact forms.  More precisely, we expect that there is a trace-free Hermitian tensor $\omega_{\alpha\bar\beta}$ such that $e^{(n-1)\Upsilon}\homega_{\alpha\bar\beta}=\omega_{\alpha\bar\beta}$ and
\begin{equation}
 \label{eqn:Iprime-transformation-goal}
 e^{(n+1)\Upsilon}\hmI_\Phi^\prime = \mI_\Phi^\prime + 2\Real\nabla^\gamma\left(\omega_\gamma{}^\beta\Upsilon_\beta\right)
\end{equation}
for all pseudo-Einstein contact forms $\theta$ and $\htheta=e^\Upsilon\theta$ on $(M^{2n+1},T^{1,0})$.  By \cref{eqn:I-prime-as-limit}, one may formally think of $\omega_{\alpha\bar\beta}$ as the limit of $\frac{1}{d-n}\mU_{\alpha\bar\beta}$ as $d\to n$.  By \cref{thm:mI-transformation}, the transformation formula of \cref{eqn:Iprime-transformation-goal} is equivalent to asking that the real $(2n-1)$-form
\[ \omega := i\omega_{\alpha\bar\beta}\theta \wedge \theta^\alpha \wedge \theta^{\bar\beta} \wedge d\theta^{n-2} \]
is such that
\[ -(n-1)\db \omega = X_\alpha\theta\wedge\theta^\alpha\wedge d\theta^{n-1} , \]
where $\db\omega:=i\nabla_\gamma\omega_{\alpha\bar\beta}\theta\wedge\theta^\gamma\wedge\theta^\alpha\wedge\theta^{\bar\beta}\wedge d\theta^{n-2}$.  This conclusion has an interpretation in terms of the bigraded Rumin complex~\cite{Garfield2001,GarfieldLee1998} which is stronger than the fact, established in the proof of \cref{thm:secondary_invariant}, that $[\xi^\Phi]=0$ in $H^{2n}(M;\bR)$.

Suppose that the real $(2n-1)$-form $\omega$ exists.  If there is a natural $2n$-form
\[ \zeta := \zeta_\alpha\theta\wedge\theta^\alpha\wedge d\theta^{n-1} \]
such that
\[ \hzeta = \zeta + \dbbar\Upsilon\wedge\omega \]
for all pseudo-Einstein contact forms $\theta$ and $\htheta=e^\Upsilon\theta$, then $\mI_\Phi^\prime-2(n-1)\Real\nabla^\gamma\zeta_\gamma$ is a local secondary invariant in the sense of \cref{conj:strong_hirachi}.  We do not expect that $\omega$ and $\zeta$, if they exist, are natural.  Instead, we hope that they can be canonically defined in terms of a pseudo-Einstein contact form.

The previous two paragraphs are pure speculation, intended to suggest a path towards better understanding the $\mI_\Phi^\prime$-curvatures and \cref{conj:weak_hirachi}.  We conclude by proving that the $\mI_\Phi^\prime$-curvatures are not local secondary invariants, and thus providing further justification for the speculations above.

\begin{proposition}
 Let $(M,T^{1,0},\theta)$ and $\Phi$ be as in \cref{thm:hirachi_ellipsoid}.  Then $\mI_\Phi^\prime$ is not a local secondary invariant in the sense of \cref{defn:local_secondary_invariant}.
\end{proposition}

\begin{proof}
 Note that, since $X_\alpha^\Phi$ is a CR invariant, it suffices to find a CR manifold $(M^{2n+1},T^{1,0})$ which admits a pseudo-Einstein contact form and also admits functions $u,v\in\mP$ such that $\int u\Real X_\alpha^\Phi v^\alpha\not=0$.  We accomplish this by computing
 \[ D := \left.\frac{d^n}{dt^n}\right|_{t=0} \int_{S^{2n+1}} u\Real (X^t)_\alpha^\Phi u^\alpha \, \theta_t\wedge d\theta_t, \]
 where $(S^{2n+1},T^{1,0},\theta_t)$ is as in \cref{thm:hirachi_ellipsoid} and $u=2\Real w$.  Note that $u$ is a CR pluriharmonic function on $S^{2n + 1}$.  A straightforward computation using \cref{eqn:time-derivative-of-X} yields
 \[ \frac{1}{(n!)^2}D = -\frac{3(n+1)}{2n}\left(\frac{n+1}{n+2}\right)^n p(\varsigma) \int_{S^{2n+1}} c^{2n}u^2\,\theta\wedge d\theta^n \not= 0 . \]
 Hence $\mI_\Phi^\prime$ is not a local secondary invariant for any nonzero $t$ sufficiently close to zero.
\end{proof}

 \subsection*{Acknowledgments}
 The authors would like to thank Kengo Hirachi and Taiji Marugame for their valuable comments.  YT would also like to thank Penn State University for its kind hospitality when part of this research was carried out.

\bibliographystyle{abbrv}
\bibliography{bib}

\newcommand{\noopsort}[1]{}
\begin{thebibliography}{10}

\bibitem{Alexakis2012}
S.~Alexakis.
\newblock {\em The decomposition of global conformal invariants}, volume 182 of
  {\em Annals of Mathematics Studies}.
\newblock Princeton University Press, Princeton, NJ, 2012.

\bibitem{BarthHulekPetersVandeVen2004}
W.~P. Barth, K.~Hulek, C.~A.~M. Peters, and A.~Van~de Ven.
\newblock {\em Compact complex surfaces}, volume~4 of {\em Ergebnisse der
  Mathematik und ihrer Grenzgebiete. 3. Folge. A Series of Modern Surveys in
  Mathematics [Results in Mathematics and Related Areas. 3rd Series. A Series
  of Modern Surveys in Mathematics]}.
\newblock Springer-Verlag, Berlin, second edition, 2004.

\bibitem{BoyerGalicki2008}
C.~P. Boyer and K.~Galicki.
\newblock {\em Sasakian geometry}.
\newblock Oxford Mathematical Monographs. Oxford University Press, Oxford,
  2008.

\bibitem{Branson1985}
T.~P. Branson.
\newblock Differential operators canonically associated to a conformal
  structure.
\newblock {\em Math. Scand.}, 57(2):293--345, 1985.

\bibitem{Branson1995}
T.~P. Branson.
\newblock Sharp inequalities, the functional determinant, and the complementary
  series.
\newblock {\em Trans. Amer. Math. Soc.}, 347(10):3671--3742, 1995.

\bibitem{BransonOrsted1991b}
T.~P. Branson and B.~{\O}rsted.
\newblock Explicit functional determinants in four dimensions.
\newblock {\em Proc. Amer. Math. Soc.}, 113(3):669--682, 1991.

\bibitem{BurnsEpstein1988}
D.~M. Burns and C.~L. Epstein.
\newblock A global invariant for three-dimensional {CR}-manifolds.
\newblock {\em Invent. Math.}, 92(2):333--348, 1988.

\bibitem{BurnsEpstein1990c}
D.~M. Burns and C.~L. Epstein.
\newblock Characteristic numbers of bounded domains.
\newblock {\em Acta Math.}, 164(1-2):29--71, 1990.

\bibitem{CaseGover2013}
J.~S. Case and A.~R. Gover.
\newblock The {$P^\prime$}-operator, the {$Q^\prime$}-curvature, and the {CR}
  tractor calculus.
\newblock {\em Ann. Sc. Norm. Super. Pisa Cl. Sci.}, \noopsort{2200}accepted.

\bibitem{CaseYang2012}
J.~S. Case and P.~C. Yang.
\newblock A {P}aneitz-type operator for {CR} pluriharmonic functions.
\newblock {\em Bull. Inst. Math. Acad. Sin. (N.S.)}, 8(3):285--322, 2013.

\bibitem{Fefferman1974}
C.~Fefferman.
\newblock The {B}ergman kernel and biholomorphic mappings of pseudoconvex
  domains.
\newblock {\em Invent. Math.}, 26:1--65, 1974.

\bibitem{Garfield2001}
P.~M. Garfield.
\newblock {\em The bigraded {R}umin complex}.
\newblock ProQuest LLC, Ann Arbor, MI, 2001.
\newblock Thesis (Ph.D.)--University of Washington.

\bibitem{GarfieldLee1998}
P.~M. Garfield and J.~M. Lee.
\newblock The {R}umin complex on {CR} manifolds.
\newblock Number 1037, pages 29--36. 1998.
\newblock CR geometry and isolated singularities (Japanese) (Kyoto, 1996).

\bibitem{GoverGraham2005}
A.~R. Gover and C.~R. Graham.
\newblock C{R} invariant powers of the sub-{L}aplacian.
\newblock {\em J. Reine Angew. Math.}, 583:1--27, 2005.

\bibitem{Graham2000}
C.~R. Graham.
\newblock Volume and area renormalizations for conformally compact {E}instein
  metrics.
\newblock In {\em The {P}roceedings of the 19th {W}inter {S}chool ``{G}eometry
  and {P}hysics'' ({S}rn\'\i, 1999)}, number~63, pages 31--42, 2000.

\bibitem{GJMS1992}
C.~R. Graham, R.~Jenne, L.~J. Mason, and G.~A.~J. Sparling.
\newblock Conformally invariant powers of the {L}aplacian. {I}. {E}xistence.
\newblock {\em J. London Math. Soc. (2)}, 46(3):557--565, 1992.

\bibitem{Hirachi2014}
K.~Hirachi.
\newblock {$Q$} and {$Q$}-prime curvature in {CR} geometry.
\newblock In {\em Proceedings of the {I}nternational {C}ongress of
  {M}athematicians---{S}eoul 2014. {V}ol. {III}}, pages 257--277. Kyung Moon
  Sa, Seoul, 2014.

\bibitem{Hirachi2013}
K.~Hirachi.
\newblock {$Q$}-prime curvature on {CR} manifolds.
\newblock {\em Differential Geom. Appl.}, 33(suppl.):213--245, 2014.

\bibitem{HirachiMarugameMatsumoto2015}
K.~Hirachi, T.~Marugame, and Y.~Matsumoto.
\newblock Variation of total {$Q$}-prime curvature on {CR} manifolds.
\newblock {\em Adv. Math.}, 306:1333--1376, 2017.

\bibitem{JerisonLee1987}
D.~Jerison and J.~M. Lee.
\newblock The {Y}amabe problem on {CR} manifolds.
\newblock {\em J. Differential Geom.}, 25(2):167--197, 1987.

\bibitem{Kobayashi1990}
R.~Kobayashi.
\newblock Moduli of {E}instein metrics on a {$K3$} surface and degeneration of
  type {${\rm I}$}.
\newblock In {\em K\"{a}hler metric and moduli spaces}, volume~18 of {\em Adv.
  Stud. Pure Math.}, pages 257--311. Academic Press, Boston, MA, 1990.

\bibitem{Lee1988}
J.~M. Lee.
\newblock Pseudo-{E}instein structures on {CR} manifolds.
\newblock {\em Amer. J. Math.}, 110(1):157--178, 1988.

\bibitem{Lu1999}
P.~Lu.
\newblock K\"{a}hler-{E}instein metrics on {K}ummer threefold and special
  {L}agrangian tori.
\newblock {\em Comm. Anal. Geom.}, 7(4):787--806, 1999.

\bibitem{Marugame2016}
T.~Marugame.
\newblock Renormalized {C}hern-{G}auss-{B}onnet formula for complete
  {K}\"{a}hler-{E}instein metrics.
\newblock {\em Amer. J. Math.}, 138(4):1067--1094, 2016.

\bibitem{Marugame2019}
T.~Marugame.
\newblock Renormalized characteristic forms of the {C}heng--{Y}au metric and
  global {CR} invariants.
\newblock arXiv:1912.10684, \noopsort{2500}preprint.

\bibitem{ReiterSon2019}
M.~Reiter and D.~N. Son.
\newblock On the {C}hern--{M}oser--{W}eyl tensor of real hypersurfaces.
\newblock arXiv:1903.12599, \noopsort{2099}preprint.

\bibitem{RoanYau1987}
S.-S. Roan and S.-T. Yau.
\newblock On {R}icci flat {$3$}-fold.
\newblock {\em Acta Math. Sinica (N.S.)}, 3(3):256--288, 1987.

\bibitem{Rumin1994}
M.~Rumin.
\newblock Formes diff\'{e}rentielles sur les vari\'{e}t\'{e}s de contact.
\newblock {\em J. Differential Geom.}, 39(2):281--330, 1994.

\bibitem{Sunada1978}
T.~Sunada.
\newblock Holomorphic equivalence problem for bounded {R}einhardt domains.
\newblock {\em Math. Ann.}, 235(2):111--128, 1978.

\bibitem{Takeuchi2018}
Y.~Takeuchi.
\newblock Ambient constructions for {S}asakian {$\eta$}-{E}instein manifolds.
\newblock {\em Adv. Math.}, 328:82--111, 2018.

\bibitem{Takeuchi2020}
Y.~Takeuchi.
\newblock A constraint on {C}hern classes of strictly pseudoconvex {CR}
  manifolds.
\newblock {\em SIGMA Symmetry Integrability Geom. Methods Appl.}, 16:005, 5
  pages, 2020.

\end{thebibliography}
\end{document}